\documentclass{amsart}

\usepackage{amsfonts}
\usepackage{graphicx}
\usepackage[onehalfspacing]{setspace}
\usepackage[ngerman, english]{babel}
\usepackage{datetime}
\usepackage{enumitem} 
\usepackage{mathtools}
\usepackage{mathrsfs}
\usepackage{cite}
\usepackage{url}
\usepackage[a4paper,margin=20ex]{geometry}
\usepackage{graphicx}
\usepackage{xpatch}
\usepackage{amsthm}
\usepackage{mleftright}
\usepackage{microtype}
\usepackage{xcolor}
\usepackage{color}
\definecolor{citation}{rgb}{0.2,0.58,0.2} 
\definecolor{formula}{rgb}{0,0,1}

\usepackage[colorlinks=true,linkcolor=formula,urlcolor=url,citecolor=citation]{hyperref}
\allowdisplaybreaks[1]

\xpatchcmd{\proof}
  {\itshape}
  {\bfseries}
  {}
  {}

\DeclarePairedDelimiterX\Set[1]{\lbrace}{\rbrace}%
 {  #1 }

\makeatletter
\def\bign#1{\mathclose{\hbox{$\left#1\vbox to8.5\p@{}\right.\n@space$}}\mathopen{}}
\def\Bign#1{\mathclose{\hbox{$\left#1\vbox to8.5\p@{}\right.\n@space$}}\mathopen{}}
\def\biggn#1{\mathclose{\hbox{$\left#1\vbox to14.5\p@{}\right.\n@space$}}\mathopen{}}
\def\Biggn#1{\mathclose{\hbox{$\left#1\vbox to17.5\p@{}\right.\n@space$}}\mathopen{}}
\makeatother

\newcommand\blfootnote[1]{%
  \begingroup
  \renewcommand\thefootnote{}\footnote{#1}%
  \addtocounter{footnote}{-1}%
  \endgroup
}

\makeatletter
\newenvironment{proof*}[1][\proofname]{\par
  \pushQED{\qed}%
  \normalfont \partopsep=\z@skip \topsep=\z@skip
  \trivlist
  \item[\hskip\labelsep
        \itshape
    #1\@addpunct{.}]\ignorespaces
}{%
  \popQED\endtrivlist\@endpefalse
}
\makeatother
\makeatletter
\newcommand*{\rom}[1]{\expandafter\@slowromancap\romannumeral #1@}
\makeatother
\def\Xint#1{\mathchoice
{\XXint\displaystyle\textstyle{#1}}%
{\XXint\textstyle\scriptstyle{#1}}%
{\XXint\scriptstyle\scriptscriptstyle{#1}}%
{\XXint\scriptscriptstyle\scriptscriptstyle{#1}}%
\!\int}
\def\XXint#1#2#3{{\setbox0=\hbox{$#1{#2#3}{\int}$ }
\vcenter{\hbox{$#2#3$ }}\kern-.582\wd0}}

\def\dashint{\Xint-}

\usepackage{amssymb}
\newtheorem*{defin}{Definition}
\newtheorem{thm}{Theorem}[section]

\newtheorem{lem}[thm]{Lemma}
\newtheorem{rem}[thm]{Remark}

\newtheorem{prop}[thm]{Proposition}

\title{Higher H\"older regularity for nonlocal equations with irregular kernel}
\author{Simon Nowak}
\address{Universit\"at Bielefeld, Fakult\"at f\"ur Mathematik, Postfach 100131, D-33501 Bielefeld, Germany}
\email{simon.nowak@uni-bielefeld.de}
\begin{document}

\maketitle

\begin{abstract}
We study the higher H\"older regularity of local weak solutions to a class of nonlinear nonlocal elliptic equations with kernels that satisfy a mild continuity assumption. An interesting feature of our main result is that the obtained regularity is better than one might expect when considering corresponding results for local elliptic equations in divergence form with continuous coefficients. Therefore, in some sense our result can be considered to be of purely nonlocal type, following the trend of various such purely nonlocal phenomena observed in recent years. Our approach can be summarized as follows. First, we use certain test functions that involve discrete fractional derivatives in order to obtain higher H\"older regularity for homogeneous equations driven by a locally translation invariant kernel, while the global behaviour of the kernel is allowed to be more general. This enables us to deduce the desired regularity in the general case by an approximation argument.
\end{abstract}
\pagestyle{headings}

\section{Introduction} 
\subsection{Basic setting and main result}
In this work, we study the higher H\"older regularity of solutions to nonlinear nonlocal equations of the form \blfootnote{Supported by SFB 1283 of the German Research Foundation.}
\begin{equation} \label{nonlocaleq}
L_A^\Phi u =  f \text{ in } \Omega \subset \mathbb{R}^n
\end{equation}
driven by a kernel that potentially exhibits a very irregular behaviour. More precisely, by modifying an approach introduced in \cite{BLS}, we prove that so-called local weak solutions to such equations are locally H\"older continuous with some explicitly determined H\"older exponent.
Here $s \in (0,1)$, $\Omega \subset \mathbb{R}^n$ is a domain (= open set), $f:\mathbb{R}^n \to \mathbb{R}$ is a given function and 
$$ L_A^\Phi u(x) := 2 \lim_{\varepsilon \to 0} \int_{\mathbb{R}^n \setminus B_\varepsilon(x)} \frac{A(x,y)}{|x-y|^{n+2s}} \Phi(u(x)-u(y))dy, \quad x \in \Omega,$$
is a nonlocal operator. Throughout the paper, for simplicity we assume that $n>2s$. Furthermore, the function $A:\mathbb{R}^n \times \mathbb{R}^n \to \mathbb{R}$ is measurable and we assume that there exists a constant $\lambda \geq 1$ such that
\begin{equation} \label{eq1}
\lambda^{-1} \leq A(x,y) \leq \lambda \text{ for almost all } x,y \in \mathbb{R}^n.
\end{equation}
Moreover, we require $A$ to be symmetric, i.e.
\begin{equation} \label{symmetry}
A(x,y)=A(y,x) \text{ for almost all } x,y \in \mathbb{R}^n.
\end{equation}
We call such a function $A$ a kernel coefficient. We define $\mathcal{L}_0(\lambda)$ as the class of all such measurable kernel coefficients $A$ that satisfy the conditions (\ref{eq1}) and (\ref{symmetry}).
Moreover, in our main results $\Phi:\mathbb{R} \to \mathbb{R}$ is assumed to be a continuous function satisfying $\Phi(0)=0$ and the following Lipschitz continuity and monotonicity assumptions, namely
\begin{equation} \label{PhiLipschitz}
|\Phi(t)-\Phi(t^\prime)| \leq \lambda |t-t^\prime| \text{ for all } t,t^\prime \in \mathbb{R}
\end{equation}
and
\begin{equation} \label{PhiMonotone}
\left (\Phi(t)-\Phi(t^\prime) \right )(t-t^\prime) \geq \lambda^{-1} (t-t^\prime)^2 \text{ for all } t,t^\prime \in \mathbb{R},
\end{equation}
where for simplicity we use the same constant $\lambda \geq 1$ as in (\ref{eq1}). In particular, if $\Phi(t)=t$, then the operator $L_A^\Phi$ reduces to a linear nonlocal operator which is widely considered in the literature. The above conditions are for example satisfied by any $C^1$ function $\Phi$ with $\Phi(0)=0$ such that the image of the first derivative $\Phi^\prime$ of $\Phi$ is contained in $[\lambda^{-1},\lambda]$. \par
Define the fractional Sobolev space
$$W^{s,2}(\Omega):= \left \{u \in L^2(\Omega) \mathrel{\Big|} \int_{\Omega} \int_{\Omega} \frac{|u(x)-u(y)|^2}{|x-y|^{n+2s}}dy < \infty \right \}$$ and denote by $W^{s,2}_{loc}(\Omega)$ the set of all functions $u \in L^2_{loc}(\Omega)$ that belong to $W^{s,2}(\Omega^\prime)$ for any relatively compact open subset $\Omega^\prime$ of $\Omega$.
In addition, we define the tail space
$$L^1_{2s}(\mathbb{R}^n):= \left \{u \in L^1_{loc}(\mathbb{R}^n) \mathrel{\Big|} \int_{\mathbb{R}^n} \frac{|u(y)|}{1+|y|^{n+2s}}dy < \infty \right \}.$$
We remark that for any function $u \in L^1_{2s}(\mathbb{R}^n)$, the quantity 
$$ \int_{\mathbb{R}^n \setminus B_{R}(x_0)} \frac{|u(y)|}{|x_0-y|^{n+2s}}dy$$
is finite for all $R>0$, $x_0 \in \mathbb{R}^n$.
For all measurable functions $u,\varphi:\mathbb{R}^n \to \mathbb{R}$, we define
$$ \mathcal{E}_A^\Phi(u,\varphi) := \int_{\mathbb{R}^n} \int_{\mathbb{R}^n} \frac{A(x,y)}{|x-y|^{n+2s}} \Phi(u(x)-u(y))(\varphi(x)-\varphi(y))dydx,$$
provided that the above expression is well-defined and finite. This is for example the case if $u \in W^{s,2}_{loc}(\Omega) \cap L^1_{2s}(\mathbb{R}^n)$ and $\varphi \in W_c^{s,2}(\Omega)$, where by $W_c^{s,2}(\Omega)$ we denote the set of all functions that belong to $W^{s,2}(\Omega)$ and are compactly supported in $\Omega$. \par
In the literature, various types of weak solutions with varying generality are considered. In this paper, we adopt the following very general notion of local weak solutions which is for example used in \cite{BL} and \cite{BLS}.
\begin{defin}
	Let $f \in L^\frac{2n}{n+2s}_{loc}(\Omega)$. We say that $u \in W^{s,2}_{loc}(\Omega) \cap L^1_{2s}(\mathbb{R}^n)$ is a local weak solution of the equation $L_A^\Phi u =f$ in $\Omega$, if 
	\begin{equation} \label{weaksolx}
	\mathcal{E}_A^\Phi(u,\varphi) =  (f,\varphi)_{L^2(\Omega)} \quad \forall \varphi \in W_c^{s,2}(\Omega).
	\end{equation}
\end{defin}
We remark that the right-hand side of (\ref{weaksolx}) is finite by the fractional Sobolev embedding (cf. \cite[Theorem 6.5]{Hitch}).
It is noteworthy that the above notion of local weak solutions contains most other notions of weak solutions considered in the literature, such as the ones considered in e.g. \cite{finnish} or \cite{Me}.

In our first main result, we are going to impose an additional continuity assumption on $A$. Namely, we assume that there exists some small $\varepsilon>0$ such that
\begin{equation} \label{contkernel}
	\lim_{h \to 0} \sup_{\substack{_{x,y \in K}\\{|x-y| \leq \varepsilon}}} |A(x+h,y+h)-A(x,y)| =0 \quad \text{for any compact set } K \subset \Omega.
\end{equation}
In particular, the condition (\ref{contkernel}) is satisfied if $A$ is either continuous close to the diagonal in $\Omega \times \Omega$ or if $A$ belongs to the following subclass of $\mathcal{L}_0(\lambda)$ which plays an important role in our proof of the desired regularity.
\begin{defin}
	Let $\Omega$ be a domain and $\lambda\geq 1$. We say that a kernel coefficient $A_0 \in \mathcal{L}_0(\lambda)$ belongs to the class $\mathcal{L}_1(\lambda,\Omega)$, if there exists a measurable function $a:\mathbb{R}^n \to \mathbb{R}$ such that $A_0(x,y)=a(x-y)$ for all $x,y \in \Omega$.
\end{defin}
A kernel coefficient that belongs to the class $\mathcal{L}_1(\lambda,\Omega)$ can be thought of being translation invariant, but only inside of $\Omega$. We also call such a kernel coefficient locally translation invariant. We note that the condition (\ref{contkernel}) is also satisfied by some more general choices of kernel coefficients, for example if
$$A(x,y)=A^\prime(x,y)A_0(x,y),$$
where $A^\prime \in \mathcal{L}_0(\lambda^\frac{1}{2})$ is continuous near the diagonal in $\Omega \times \Omega$ and $A_0$ belongs to the class $\mathcal{L}_1(\lambda^\frac{1}{2},\Omega)$, but is not required to satisfy any continuity or smoothness assumption. Moreover, we stress that the condition given by (\ref{contkernel}) only restricts the behaviour of $A$ close to the diagonal in $\Omega \times \Omega$, while away from the diagonal in $\Omega \times \Omega$ and outside of $\Omega \times \Omega$ a more general behaviour is possible. \par
We are now in the position to state our main results.
\begin{thm} \label{C2sreg}
	Let $\Omega \subset \mathbb{R}^n$ be a domain, $s \in (0,1)$, $\lambda \geq 1$ and $f \in L^q_{loc}(\Omega)$ for some $q>\frac{n}{2s}$. Consider a kernel coefficient $A \in \mathcal{L}_0(\lambda)$ that satisfies the condition (\ref{contkernel}) for some $\varepsilon>0$ and suppose that $\Phi$ satisfies (\ref{PhiLipschitz}) and (\ref{PhiMonotone}) with respect to $\lambda$. Moreover, assume that $u \in W^{s,2}_{loc}(\Omega) \cap L^1_{2s}(\mathbb{R}^n)$ is a local weak solution of the equation $L_{A}^\Phi u = f$ in $\Omega$. Then for any $0<\alpha<\min \big \{2s-\frac{n}{q},1 \big\}$, we have $u \in C^\alpha_{loc}(\Omega)$. \newline Furthermore, for all $R>0$, $x_0 \in \Omega$ such that $B_R(x_0) \Subset \Omega$ and any $\sigma \in (0,1)$, we have
	\begin{equation} \label{Hoeldest}
	\begin{aligned}
	[u]_{C^\alpha(B_{\sigma R}(x_0))} \leq \frac{C}{R^\alpha} & \bigg ( R^{-\frac{n}{2}} ||u||_{L^2(B_R(x_0))} + R^{2s} \int_{\mathbb{R}^n \setminus B_{R}(x_0)} \frac{|u(y)|}{|x_0-y|^{n+2s}}dy \\ & + R^{2s-\frac{n}{q}} ||f||_{L^q(B_R(x_0))} \bigg ),
	\end{aligned}
	\end{equation}
	where $C=C(n,s,\lambda,\alpha,q,\sigma,\varepsilon)>0$ and 
	$$[u]_{C^{\alpha}(B_{\sigma R}(x_0))}:=\sup_{\substack{_{x,y \in B_{\sigma  R}(x_0)}\\{x \neq y}}} \frac{|u(x)-u(y)|}{|x-y|^{\alpha}}.$$
\end{thm}
If we focus on obtaining H\"older regularity for some fixed exponent $0<\alpha<\min \big \{2s-\frac{n}{q},1 \big\}$, then we can slightly weaken the assumption on $A$ as follows. Roughly speaking, in this case it is enough to require that $A$ is locally close enough to being translation invariant, while the condition (\ref{contkernel}) essentially means that $A$ is locally arbitrarily close to being translation invariant. This slight "room for error" is typical when one uses approximation techniques in order to obtain regularity results, see for example \cite{CSa}.
\begin{thm} \label{C2srega}
Let $\Omega \subset \mathbb{R}^n$ be a domain, $s \in (0,1)$, $\lambda \geq 1$ and $f \in L^q_{loc}(\Omega)$ for some $q>\frac{n}{2s}$. Consider a kernel coefficient $A \in \mathcal{L}_0(\lambda)$ and suppose that $\Phi$ satisfies (\ref{PhiLipschitz}) and (\ref{PhiMonotone}) with respect to $\lambda$. Fix some $0<\alpha<\min \big \{2s-\frac{n}{q},1 \big\}$. Then there exists some small enough $\delta=\delta(\alpha,n,s,\lambda,q)>0$, such that if for any $z \in \Omega$, there exists some small enough radius $r_{z}>0$ and some $A_{z} \in \mathcal{L}^1(\lambda,B_{r_{z}}(z))$ such that $$||A-A_{z}||_{L^\infty(B_{r_{z}}(z) \times B_{r_{z}}(z))} \leq \delta,$$
then for any local weak solution $u \in W^{s,2}_{loc}(\Omega) \cap L^1_{2s}(\mathbb{R}^n)$ of the equation $L_{A}^\Phi u = f$ in $\Omega$, we have $u \in C^\alpha_{loc}(\Omega)$. Moreover, for all $R>0$, $x_0 \in \Omega$ such that $B_R(x_0) \Subset \Omega$ and any $\sigma \in (0,1)$, $u$ satisfies the estimate (\ref{Hoeldest}) with respect to $\alpha$ and some constant $C=C(n,s,\lambda,\alpha,q,\sigma,\{r_z\}_{z \in \Omega})>0$.
\end{thm}

\begin{rem} \normalfont
In order to provide some context, let us briefly consider the local elliptic equation in divergence form of the type
\begin{equation} \label{localeq}
\textnormal{div}(B \nabla u)=0 \quad \text{in } \Omega,
\end{equation}
where the matrix of coefficients $B=\{b_{ij}\}_{i,j=1}^n$ is assumed to be uniformly elliptic and bounded. The equation (\ref{localeq}) can in some sense be thought of as a local analogue of the nonlocal equation (\ref{nonlocaleq}) corresponding to the limit case $s=1$. A classical regularity result states that if the coefficients $b_{ij}$ are continuous, then weak solutions $u \in W^{1,2}_{loc}(\Omega)$ of the equation (\ref{localeq}) are locally H\"older continuous for any exponent $\alpha \in (0,1)$, see for example \cite[Corollary 5.18]{Giaq}. Heuristically, one might therefore expect that the optimal regularity in the setting of nonlocal equations with continuous kernel coefficient should not exceed $C^s$ regularity. Nevertheless, Theorem \ref{C2sreg} in particular shows that weak solutions to nonlocal equations of the type $L_A^\Phi u=0$ in $\Omega$ are locally $C^\alpha$ for any $0<\alpha<\min \big \{2s,1 \big\}$ whenever $A \in \mathcal{L}_0(\lambda)$ is continuous, exceeding $C^s$ regularity. In particular, in the case when $s \geq 1/2$, weak solutions to homogeneous nonlocal equations with continuous kernel coefficients enjoy the same amount of H\"older regularity as weak solutions to corresponding local equations with continuous coefficients, despite the fact that the order of such nonlocal equations is lower. \par Such at first sight unexpected additional regularity is however not untypical in the context of nonlocal equations and has been observed in various previous works in the context of Sobolev regularity. For example, in \cite{selfimpro} and \cite{Schikorra} it is shown that already in the setting of a general kernel coefficient $A \in \mathcal{L}_0(\lambda)$, weak solutions to nonlocal equations of the type (\ref{nonlocaleq}) are slightly higher differentiable than initially assumed along the scale of Sobolev spaces, which is a phenomenon not shared by local elliptic equations of the type (\ref{localeq}) with coefficients that are merely measurable. \newline
Another result in this direction was recently proved in \cite{MSY}, where the authors in particular show that if $A \in \mathcal{L}_0(\lambda)$ is H\"older continuous with some arbitrary H\"older exponent and $\Phi(t)=t$, then weak solutions of the equation $L_A^\Phi u=0$ in $\mathbb{R}^n$ belong to $W^{\alpha,p}_{loc}(\mathbb{R}^n)$ for any $\alpha<\min \big \{2s,1 \big\}$ and any $2 \leq p<\infty$, while for local equations of the type (\ref{localeq}) with corresponding H\"older continuous coefficients no comparable gain in differentiability is achievable. In particular, by the Sobolev embedding this result implies that such weak solutions belong to $C^\alpha_{loc}(\mathbb{R}^n)$ for any $0<\alpha<\min \big \{2s,1 \big\}$, which is consistent with our main result. Our main result shows that this amount of higher H\"older regularity is also enjoyed by local weak solutions of possibly nonlinear equations driven by kernel coefficients of class $\mathcal{L}_0(\lambda)$ that satisfy the continuity assumption (\ref{contkernel}).
\end{rem}
\begin{rem} \label{Sobrem} \normalfont
Besides being interesting for its own sake, one of our main motivations is that Theorem \ref{C2sreg} also has some interesting potential applications concerning the Sobolev regularity of solutions to nonlocal equations. A first such application can briefly be summarized as follows. In \cite{Me}, in the main result it is assumed that $A$ is globally translation invariant, i.e. that $A$ belongs to the class $\mathcal{L}_1(\lambda,\mathbb{R}^n)$. However, this assumption is only used in order to ensure that the H\"older estimate (\ref{Hoeldest}) from Theorem \ref{C2sreg} is valid, which up to this point was only known for translation invariant kernels, cf. \cite[Theorem 4.6]{Me}. Since otherwise the proofs in \cite{Me} only rely on the properties (\ref{eq1}) and (\ref{symmetry}) of $A$, from Theorem \ref{C2sreg} above we conclude that the statement of \cite[Theorem 1.1]{Me} is also true for general kernel coefficients $A$ of class $\mathcal{L}_0(\lambda)$ that satisfy the condition (\ref{contkernel}).
\end{rem}
\subsection{Approach and previous results}
As mentioned, our approach is strongly influenced by an approach introduced in \cite{BLS}, where a similar result concerning higher H\"older regularity is proved for the fractional $p$-Laplacian in the superquadratic case when $p \geq 2$. Although for simplicity we restrict ourselves to the quadratic case when $p=2$, in contrast to \cite{BLS} we deal with a nonlinearity already in the quadratic setting and most importantly, we also treat equations driven by general kernel coefficients $A$ that satisfy the mild assumption (\ref{contkernel}), while in \cite{BLS} only the case when $A \equiv 1$ is considered. Also, we stress that by combining our techniques with some more techniques from \cite{BLS}, our approach could be modified in order to treat also nonlinearities with nonlinear growth of the type $\Phi(t) \approx t^{p-1}$. However, since the additional difficulties arising from such a generalization were already dealt with in \cite{BLS} and we instead want to focus on the difficulties arising from considering equations with general coefficients, we decided not to pursue this direction in this work. \par Let us briefly summarize our approach, highlighting the differences to the one used in \cite{BLS}. First, we prove the higher H\"older regularity for homogeneous equations driven by a locally translation invariant kernel coefficient, see section 3. As in \cite{BLS}, the main idea in this case is to test the equation with certain monotone power functions of discrete fractional derivatives leading to an incremental higher integrability and differentiability result on the scale of certain Besov-type spaces. However, in our setting we also need to carefully use the local translation invariance and the bounds imposed on $A$, and also the assumptions (\ref{PhiLipschitz}) and (\ref{PhiMonotone}) imposed on $\Phi$ in order to overcome the difficulties that arise due to the presence of the general kernel and the general type of nonlinearity. Moreover, we remark that restricting ourselves to equations with linear growth has the advantage that the proof of this incremental higher regularity result simplifies quite substantially in some other respects. The obtained incremental gain in regularity is then iterated, in order for the desired H\"older regularity to follow by embedding. \par
In section 4, we then treat the general case of inhomogeneous equations driven by a kernel coefficient satisfying the condition (\ref{contkernel}) by an approximation argument. In the corresponding approximation argument applied in \cite{BLS}, the solution is approximated by a solution of a corresponding equation with zero right-hand side, while the nonlocal operator driving the equation is left unchanged. In order to be able to treat equations with a general kernel coefficient $A$ of class $\mathcal{L}_0(\lambda)$ that satisfies only the continuity assumption (\ref{contkernel}), in addition to freezing the right-hand side, we also need to locally replace $A$ by a corresponding locally translation invariant kernel coefficient, which is possible in view of the assumption (\ref{contkernel}).
Since by the first part of the proof the desired H\"older regularity is already known for solutions to equations with locally translation invariant kernel coefficients, we can then transfer this regularity from the approximate solution to the solution itself. In other words, in some sense we locally freeze the coefficient, in order to transfer the regularity from an equation for which the higher regularity can be proved directly to an equation driven by a less regular kernel. This strategy can be thought of as a nonlocal counterpart of corresponding techniques widely used in the study of higher regularity for local elliptic equations, although we stress that in our nonlocal setting we have to overcome a number of additional difficulties which are not present in the local setting in order to execute such an approximation argument successfully. Moreover, we believe that just like in the local setting, the approximation techniques developed in this paper are flexible enough in order to be adaptable to also proving other higher regularity results for nonlocal equations similar to (\ref{nonlocaleq}). \par
Regarding other related regularity results, in \cite{Fall} a similar result is proved in the linear case when $\Phi(t)=t$, where $A$ is required to be locally close enough to $b \left (\frac{x-y}{|x-y|} \right )$ for some even function $b:S^{n-1} \to \mathbb{R}$ that is bounded between two positive constants, which is contained in our assumption on $A$ in Theorem \ref{C2srega}. More results concerning higher H\"older regularity for various types of nonlocal equations are for instance contained in \cite{Fall1}, \cite{NonlocalGeneral}, \cite{CSa}, \cite{Stinga} and \cite{Grubb}. Furthermore, results regarding basic H\"older regularity for nonlocal equations are proved for example in \cite{finnish}, \cite{Kassmann}, \cite{Silvestre} and \cite{Peral}, while results concerning Sobolev regularity can be found for example in \cite{selfimpro}, \cite{Schikorra}, \cite{BL}, \cite{Cozzi}, \cite{MSY}, \cite{DongKim} and \cite{Me}. Finally, for some regularity results concerning nonlocal equations similar to (\ref{nonlocaleq}) in the more general setting of measure data, we refer to \cite{mdata}.

\section{Preliminaries}
\subsection{Some notation}
Let us fix some notation which we use throughout the paper. By $C$, $c$, $C_i$ and $c_i$, $i \in \mathbb{N}_0$, we always denote positive constants, while dependences on parameters of the constants will be shown in parentheses. As usual, by
$$ B_r(x_0):= \{x \in \mathbb{R}^n \mid |x-x_0|<r \}, \quad \overline B_r(x_0):= \{x \in \mathbb{R}^n \mid |x-x_0| \leq r \}$$
we denote the open and closed ball with center $x_0 \in \mathbb{R}^n$ and radius $r>0$, respectively. Moreover, if $E \subset \mathbb{R}^n$ is measurable, then by $|E|$ we denote the $n$-dimensional Lebesgue-measure of $E$. If $0<|E|<\infty$, then for any $u \in L^1(E)$ we define
$$ \overline u_{E}:= \dashint_{E} u(x)dx := \frac{1}{|E|} \int_{E} u(x)dx.$$
Next, for any $p \in (1,\infty)$ we define the function $J_p:\mathbb{R} \to \mathbb{R}$ by
$$J_p(t):=|t|^{p-2}t.$$
Moreover, for any measurable function $\psi:\mathbb{R}^n \to \mathbb{R}$ and any $h \in \mathbb{R}^n$, we define
$$ \psi_h(x):=\psi(x+h), \quad \delta_h \psi(x):=\psi_h(x)-\psi(x), \quad \delta_h^2(x):= \delta_h (\delta_h \psi(x))=\psi_{2h}(x)+\psi(x)-2\psi_h(x).$$

\subsection{The nonlocal tail}
In this section, for convenience we state and proof the following two simple results concerning the nonlocal tail of a function which we use frequently throughout the paper.
\begin{lem} \label{tailestz}
	Let $s \in (0,1)$ and $0<r<R$. Then for any $x \in \overline B_r$ and any $u \in L^1_{2s}(\mathbb{R}^n)$, we have
	$$ \int_{\mathbb{R}^n \setminus B_R} \frac{|u(y)|}{|x-y|^{n+2s}}dy \leq \left (\frac{R}{R-r} \right )^{n+2s} \int_{\mathbb{R}^n \setminus B_R} \frac{|u(y)|}{|y|^{n+2s}}dy.$$
\end{lem}
\begin{proof}
	The claim follows directly from the observation that for any $x \in \overline B_r$ and any $y \in \mathbb{R}^n \setminus B_R$, we have
	$$ |y| \leq |x-y|+|x| = |x-y| \left ( 1+\frac{|x|}{|x-y|} \right ) \leq |x-y| \left ( 1+\frac{r}{R-r} \right ) = \frac{R}{R-r} |x-y|. $$
\end{proof}

\begin{lem} \label{tail}
	Let $s \in (0,1)$, $r>0$ and $x_0 \in B_1$ such that $B_r(x_0) \subset B_1$. Then for any $u \in L^1_{2s}(\mathbb{R}^n)$, we have
	$$ \int_{\mathbb{R}^n \setminus B_r(x_0)} \frac{|u(y)|}{|x_0-y|^{n+2s}}dy \leq r^{-(n+2s)} \left ( ||u||_{L^1(B_1)}+ \int_{\mathbb{R}^n \setminus B_1} \frac{|u(y)|}{|y|^{n+2s}}dy \right ).$$
\end{lem}
\begin{proof}
	Since by assumption $x_0 \in \overline B_{1-r}$, with the help of Lemma \ref{tailestz} we obtain
	\begin{align*}
	\int_{\mathbb{R}^n \setminus B_r(x_0)} \frac{|u(y)|}{|x_0-y|^{n+2s}}dy = & \int_{B_1 \setminus B_r(x_0)} \frac{|u(y)|}{|x_0-y|^{n+2s}}dy + \int_{\mathbb{R}^n \setminus B_1} \frac{|u(y)|}{|x_0-y|^{n+2s}}dy \\
	\leq & r^{-(n+2s)} ||u||_{L^1(B_1)} + r^{-(n+2s)} \int_{\mathbb{R}^n \setminus B_1} \frac{|u(y)|}{|y|^{n+2s}}dy,
	\end{align*}
	which finishes the proof.
\end{proof}

\subsection{The fractional Sobolev space $W^{s,2}$}
First of all, for notational convenience for any domain $\Omega \subset \mathbb{R}^n$ we define the seminorm associated to the space $W^{s,2}(\Omega)$ by
$$[u]_{W^{s,2}(\Omega)} := \left (\int_{\Omega} \int_{\Omega} \frac{|u(x)-u(y)|^2}{|x-y|^{n+2s}}dydx \right )^{1/2}, $$
so that we have
$$W^{s,2}(\Omega)= \left \{u \in L^2(\Omega) \mid [u]_{W^{s,2}(\Omega)} < \infty \right \}.$$
Moreover, we define the space 
$$W^{s,2}_0(\Omega):= \left \{u \in W^{s,2}(\mathbb{R}^n) \mid u \equiv 0 \text{ in } \mathbb{R}^n \setminus \Omega \right \}.$$
The following Poincar\'e-type inequality associated to the space $W^{s,2}$ will frequently be used throughout the paper.
\begin{lem} \label{Friedrichs} (fractional Friedrichs-Poincar\'e inequality)
	Let $s \in(0,1)$ and consider a bounded domain $\Omega \subset \mathbb{R}^n$. For any $u \in W^{s,2}_0(\Omega)$, we have
	\begin{equation} \label{FPI4}
	\int_{\Omega} |u(x)|^2 dx \leq C |\Omega|^{\frac{2s}{n}} \int_{\mathbb{R}^n} \int_{\mathbb{R}^n} \frac{|u(x)-u(y)|^2}{|x-y|^{n+2s}}dydx,
	\end{equation}
	where $C=C(n,s)>0$.
\end{lem}
\begin{proof}
	Since $u \in W^{s,2}_0(\Omega) \subset W^{s,2}(\mathbb{R}^n)$ and $n>2s$, applying H\"older's inequality and then the fractional Sobolev inequality (cf. \cite[Theorem 6.5]{Hitch}) leads to
	\begin{align*}
	\int_{\Omega} |u(x)|^2 dx & \leq |\Omega|^{\frac{2s}{n}} \left ( \int_{\Omega} |u(x)|^{\frac{2n}{n-2s}} dx \right )^{\frac{n-2s}{n}} \\
	& \leq C |\Omega|^{\frac{2s}{n}} \int_{\mathbb{R}^n} \int_{\mathbb{R}^n} \frac{|u(x)-u(y)|^2}{|x-y|^{n+2s}}dydx,
	\end{align*}
	where $C=C(n,s)>0$. This finishes the proof.
\end{proof}

\subsection{Besov-type spaces}
Next, let us introduce some function spaces of Besov-type. In order to do so, for $q \in [1,\infty)$ and any function $u \in L^q(\mathbb{R}^n)$ we define the quantities
$$ [u]_{\mathcal{N}_\infty^{\beta,q}(\mathbb{R}^n)}:= \sup_{|h|>0} \left | \left | \frac{\delta_h u}{|h|^\beta} \right | \right |_{L^q(\mathbb{R}^n)}, \quad 0<\beta \leq 1 $$
and 
$$ [u]_{\mathcal{B}_\infty^{\beta,q}(\mathbb{R}^n)}:= \sup_{|h|>0} \left | \left | \frac{\delta_h^2 u}{|h|^\beta} \right | \right |_{L^q(\mathbb{R}^n)}, \quad 0<\beta <2. $$
This enables us to define the two Besov-type spaces
$$ \mathcal{N}_\infty^{\beta,q}(\mathbb{R}^n):= \left \{u \in L^q(\mathbb{R}^n) \mid [u]_{\mathcal{N}_\infty^{\beta,q}(\mathbb{R}^n)} < \infty \right \}, \quad 0<\beta \leq 1 $$
and 
$$ \mathcal{B}_\infty^{\beta,q}(\mathbb{R}^n):= \left \{u \in L^q(\mathbb{R}^n) \mid [u]_{\mathcal{B}_\infty^{\beta,q}(\mathbb{R}^n)} < \infty \right \}, \quad 0<\beta <2. $$
The following embedding result can be found in \cite[Lemma 2.3]{BS}.
\begin{lem} \label{embedding5}
	Let $\beta \in (0,1)$ and $q \in [1,\infty)$. Then we have the continuous embedding $$\mathcal{B}_\infty^{\beta,q}(\mathbb{R}^n) \hookrightarrow \mathcal{N}_\infty^{\beta,q}(\mathbb{R}^n).$$
	More precisely, for every $u \in \mathcal{B}_\infty^{\beta,q}(\mathbb{R}^n)$ we have
	$$ [u]_{\mathcal{N}_\infty^{\beta,q}(\mathbb{R}^n)} \leq \frac{C}{1-\beta} [u]_{\mathcal{B}_\infty^{\beta,q}(\mathbb{R}^n)},$$
	where $C=C(n,q)>0$.
\end{lem}

We also need the following embedding result, cf. \cite[Theorem 2.8]{BLS}.
\begin{lem}\label{Holderemb}
	Let $q \in [1,\infty)$ and $\beta \in (0,1)$ such that $\beta q>n$. If $u \in \mathcal{N}_\infty^{\beta,q}(\mathbb{R}^n)$, then for any $\alpha \in (0,\beta-n/q)$ we have $u \in C^{\alpha}_{loc}(\mathbb{R}^n)$. More precisely, for every $u \in \mathcal{N}_\infty^{\beta,q}(\mathbb{R}^n)$ we have
	$$ \sup_{\substack{_{x,y \in \mathbb{R}^n}\\{x \neq y}}} \frac{|u(x)-u(y)|}{|x-y|^{\alpha}} \leq C \left ([u]_{\mathcal{N}_\infty^{\beta,q}(\mathbb{R}^n)} \right)^\frac{\alpha q +n}{\beta q} \left (||u||_{L^{q}(\mathbb{R}^n)} \right)^{1-\frac{\alpha q +n}{\beta q}},$$
	where $C=C(n,q,\alpha,\beta)>0$.
\end{lem}
Finally, the following result can be found in \cite[Proposition 2.6]{BL}.
\begin{prop} \label{diffsobolev}
	Let $s \in (0,1)$.
	\begin{itemize}
		\item Let $0<r<R$. For any function $\psi \in W^{s,2}_0(B_r)$, we have
		$$ \sup_{|h|>0} \left | \left | \frac{\delta_h \psi}{|h|^s} \right | \right |_{L^2(\mathbb{R}^n)}^2 \leq C \left (\frac{R}{r} \right)^n \left(\frac{R}{R-r} \right )^{3} [\psi]_{W^{s,2}(B_R)}^2,$$
		where $C=C(n,s)>0$.
		\item Let $\Omega \subset \mathbb{R}^n$ be an open set and $\psi \in W^{s,2}_{loc}(\Omega)$. Then for any $R>0$ such that $B_R \Subset \Omega$ and any $0<h_0 \leq \textnormal{dist}(B_R,\partial \Omega)/2$, we have
		\begin{align*}
		\sup_{|h|>0} \left | \left | \frac{\delta_h \psi}{|h|^s} \right | \right |_{L^2(B_R)}^2 \leq C ||\psi||_{W^{s,2}(B_{R+h_0})}^2,
		\end{align*}
		where $C=C(n,s,R,h_0)>0$.
	\end{itemize}
\end{prop}

\subsection{Some elementary inequalities}
The proof of the following elementary inequality can be found in \cite[Lemma A.3]{BLS}.
\begin{lem} \label{elementary0}
For all $X,Y \in \mathbb{R}$ and any $p \geq 1$, we have
$$ \left | |X|^{p-1}X-|Y|^{p-1}Y \right | \geq \frac{1}{C}|X-Y|^p,$$
where $C=C(p)>0$.
\end{lem}
Next, we prove two elementary inequalities which involve the function $J_p$ defined in section 2.1 and are based on the monotonicity property (\ref{PhiMonotone}) of $\Phi$.
\begin{lem} \label{elementary1}
	Let $q \geq 1$ and $a,b,c,d \in \mathbb{R}^n$. If $\Phi:\mathbb{R} \to \mathbb{R}$ satisfies (\ref{PhiMonotone}), then we have
	\begin{align*}
	& (\Phi(a-c)-\Phi(b-d)) \left (J_{q+1}(a-b)-J_{q+1} (c-d) \right ) \\
	& \geq \frac{1}{2}\lambda^{-1} \left | (a-b) - (c-d) \right |^2 (|a-b|^{{q-1}}+|c-d|^{{q-1}}).
	\end{align*}
\end{lem}

\begin{proof}
If $a-c=b-d$, then also $a-b=c-d$, so that in this case both sides of the inequality vanish. Next, we consider the case when $a-c \neq b-d$. 
In view of the monotonicity assumption (\ref{PhiMonotone}) imposed on $\Phi$, we have 
\begin{equation} \label{mon5}
\left (\Phi(a-c)-\Phi(b-d) \right )((a-b)-(c-d)) \geq \lambda^{-1} ((a-b)-(c-d))^2. \end{equation}
Moreover, by \cite[page 71]{Lind}, for all $x,y \in \mathbb{R}$ we have
$$ (J_{q+1}(y)-J_{q+1}(x))(y-x)= \frac{1}{2} \left(|y|^{q-1}+|x|^{q-1} \right) (y-x)^2 + \frac{|y|^{q-1}-|x|^{q-1}}{2} (y^2-x^2) .$$
Since the last term on the right-hand side is non-negative, by choosing $y=a-b$ and $x=c-d$ we obtain
$$ (J_{q+1}(a-b)-J_{q+1}(c-d))((a-b)-(c-d)) \geq \frac{1}{2} \left(|a-b|^{q-1}+|c-d|^{q-1} \right) ((a-b)-(c-d))^2 .$$
Multiplying the inequality (\ref{mon5}) with the one in the previous display leads to 
\begin{align*}
& \left (\Phi(a-c)-\Phi(b-d) \right ) (J_{q+1}(a-b)-J_{q+1}(c-d))((a-b)-(c-d))^2 \\
\geq & \frac{1}{2}\lambda^{-1} \left(|a-b|^{q-1}+|c-d|^{q-1} \right) ((a-b)-(c-d))^4,
\end{align*}
so that the claim follows by simplifying the factor $((a-b)-(c-d))^2$ from both sides.
\end{proof}

\begin{lem} \label{elementary2}
Let $q \geq 1$ and $a,b,c,d \in \mathbb{R}^n$. If $\Phi:\mathbb{R} \to \mathbb{R}$ satisfies (\ref{PhiMonotone}), then we have
\begin{align*}
& (\Phi(a-c)-\Phi(b-d)) \left (J_{q+1}(a-b)-J_{q+1} (c-d) \right ) \\
& \geq \frac{1}{C} \left | |a-b|^{\frac{q-1}{2}}(a-b) - |c-d|^{\frac{q-1}{2}}(c-d) \right |^2,
\end{align*}
where $C=C(\lambda,q)>0$.
\end{lem}
\begin{proof}
If $a-c=b-d$, then both sides of the inequality vanish. Next, let us consider the case when $a-c \neq b-d$.
In view of (\ref{PhiMonotone}), we have 
\begin{align*}
& (\Phi(a-c)-\Phi(b-d)) \left (J_{q+1}(a-b)-J_{q+1} (c-d) \right ) \\
= & (\Phi(a-c)-\Phi(b-d)) ((a-c)-(b-d)) \\
& \times \left (J_{q+1}(a-b)-J_{q+1} (c-d) \right ) ((a-b)-(c-d)) \\
& \times ((a-c)-(b-d))^{-2} \\
\geq & \lambda^{-1} \left (J_{q+1}(a-b)-J_{q+1} (c-d) \right ) ((a-b)-(c-d)).
\end{align*}
The right-hand side of the above estimate can be further estimated by applying \cite[Lemma A.1]{BLS} with $p=q+1$ and $q=2$, which yields
\begin{align*}
& \left (J_{q+1}(a-b)-J_{q+1} (c-d) \right ) ((a-b)-(c-d)) \\
\geq & q \left (\frac{2}{q+1} \right)^2  \left | |a-b|^{\frac{q-1}{2}}(a-b) - |c-d|^{\frac{q-1}{2}}(c-d) \right |^2.
\end{align*}
The claim now follows by combining the last two displays.
\end{proof}

\subsection{Some preliminary estimates}
The following Caccioppoli-type inequality can be proved in essentially the same way as the one in \cite[Theorem 3.1]{selfimpro}.
\begin{thm} \label{Cacc}
Let $0<r<R$, $x_0 \in \mathbb{R}^n$, $\lambda \geq 1$ and $f \in L^{\frac{2n}{n+2s}}(B_R(x_0))$. Moreover, assume that 
$A \in \mathcal{L}_0(\lambda)$ and that the Borel function $\Phi:\mathbb{R} \to \mathbb{R}$ satisfies
\begin{equation} \label{weakassump}
|\Phi(t)| \leq \lambda t, \quad \Phi(t)t \geq \lambda^{-1} t^2 \quad \forall t \in \mathbb{R}.
\end{equation}
Then for any local weak solution $u \in W^{s,2}(B_{R}(x_0)) \cap L^1_{2s}(\mathbb{R}^n)$ of 
$ L_A^\Phi u = f \text{ in } B_{R}(x_0),$ we have
\begin{align*}
& \int_{B_r(x_0)} \int_{B_r(x_0)} \frac{|u(x)-u(y)|^2}{|x-y|^{n+2s}}dydx \\ \leq & C \bigg ( R^{-2s} \int_{B_R(x_0)} u(x)^2dx + \int_{\mathbb{R}^n \setminus B_R(x_0)} \frac{|u(y)|}{|x_0-y|^{n+2s}}dy \int_{B_R(x_0)} |u(x)|dx \\ & + \left (\int_{B_R(x_0)} |f(x)|^{\frac{2n}{n+2s}}dx \right )^\frac{n+2s}{n} \Bigg ),
\end{align*}
where $C=C(n,s,\lambda,r,R)>0$.
\end{thm}
We remark that the assumptions in (\ref{weakassump}) are clearly implied by the assumptions $\Phi(0)=0$, (\ref{PhiLipschitz}) and (\ref{PhiMonotone}) which are used in our main results. \par
The following result on local boundedness is essentially given by \cite[Theorem 3.8]{BP}, where the below result is stated under the stronger assumption that $u \in W^{s,2}_0(B_R(x_0))$ and in setting of the fractional $p$-Laplacian, which applied to our setting means that strictly speaking it only contains the case when $\Phi(t)=t$ and $A(x,y) \equiv 1$. Nevertheless, an inspection of the proof shows that it remains valid for local weak solutions, see also \cite[Theorem 3.2]{BLS}. Moreover, the case of a general $\Phi$ and a general $A$ can easily be treated by noting that the Caccioppoli-type inequality from \cite[Proposition 3.5]{BP} remains valid for such a general $\Phi$ and a general $A$ by simply applying the bounds imposed on $\Phi$ and $A$ whenever appropriate in a similar fashion as in \cite[Theorem 3.1]{selfimpro}. Therefore, we have the following result.
\begin{thm} \label{finnish}
	Let $R>0$, $x_0 \in \mathbb{R}^n$, $\lambda \geq 1$, $\sigma \in (0,1)$ and $f \in L^q(B_R(x_0))$ for some $q>\frac{n}{2s}$. Moreover, consider a kernel coefficient $A \in \mathcal{L}_0(\lambda)$ and assume that the Borel function $\Phi:\mathbb{R} \to \mathbb{R}$ satisfies (\ref{weakassump}).
	Then for any local weak solution $u \in W^{s,2}(B_{R}(x_0)) \cap L^1_{2s}(\mathbb{R}^n)$ of the equation 
	$$ L_A^\Phi u = f \text{ in } B_{R}(x_0),$$
	we have the estimate
	\begin{align*}
	\sup_{x \in \overline B_{\sigma R}(x_0)}|u(x)| \leq & C \left ( \bigg ( \dashint_{B_{R}(x_0)}u(x)^2dx \right )^{\frac{1}{2}} + R^{2s} \int_{\mathbb{R}^n \setminus B_{\sigma R}(x_0)} \frac{|u(y)|}{|x_0-y|^{n+2s}}dy \\ & + R^{2s-\frac{n}{q}} ||f||_{L^q(B_{R}(x_0))} \bigg ),
	\end{align*}
	where $C=C(n,s,\lambda,q,\sigma)>0$.
\end{thm}
In the case when $f=0$ and $\Phi(t)=t$, the following result concerning basic H\"older regularity
follows from \cite[Theorem 1.2]{finnish}. The case of a general $\Phi$ can again be treated by replacing the Caccioppoli inequality given by \cite[Theorem 1.4]{finnish} with the one from \cite[Theorem 3.1]{selfimpro}. The result with a general right-hand side can then be proved in essentially the same way as in \cite[section 3.2]{BLS}.
\begin{thm} \label{finnish1}
	Under the same assumptions and notation as in Theorem \ref{finnish}, there exists some $\beta=\beta(n,s,\lambda,q,\sigma) \in (0,1)$ such that $u \in C^\beta(\overline B_{\sigma R}(x_0))$. Moreover, we have the estimate 
	\begin{align*}
	[u]_{C^\beta(B_{\sigma R}(x_0))} \leq & C \left ( \bigg ( \dashint_{B_{R}(x_0)}u(x)^2dx \right )^{\frac{1}{2}} + R^{2s} \int_{\mathbb{R}^n \setminus B_{\sigma R}(x_0)} \frac{|u(y)|}{|x_0-y|^{n+2s}}dy \\ & + R^{2s-\frac{n}{q}} ||f||_{L^q(B_{R}(x_0))} \bigg ),
	\end{align*}
	where $C=C(n,s,\lambda,q,\sigma,\beta)>0$.
\end{thm}

\section{Higher H\"older regularity for homogeneous equations with locally translation invariant kernel}
\subsection{Incremental higher integrability and differentiability}
The key ingredient to proving the desired higher H\"older regularity for homogeneous equations with locally translation invariant kernel is provided by the following incremental higher integrability and differentiability result on the scale of Besov-type spaces. In the case of the fractional $p$-Laplacian for $p \geq 2$, the below result was proved in \cite[Proposition 5.1]{BLS}. Besides the fact that we treat equations with arbitrary locally translation invariant kernels, it is also interesting that in our setting of equations with linear growth, we are able to directly prove both higher integrability and differentiability, while for possibly degenerate equations as in \cite{BLS} it is necessary to first obtain a pure higher integrability result (cf. \cite[Proposition 4.1]{BLS}), which is then used in order to also obtain higher differentiability. We remark that this additional higher differentiability does not seem to have a counterpart in the context of local equations and is one of the main reasons why in our nonlocal setting we are able to exceed $C^s$ regularity. \newline Moreover, note that although at this point we work with solutions that are bounded, this assumption will later be removed by using Theorem \ref{finnish}. 
\begin{prop} \label{increment}
Let $u \in W^{s,2}(B_1) \cap L^1_{2s}(\mathbb{R}^n)\cap L^\infty(B_1)$ be a local weak solution of
\begin{equation} \label{eq871}
L_A^\Phi u=0 \text{ in } B_1,
\end{equation}
where $A \in \mathcal{L}_1(\lambda,B_1)$ and $\Phi$ satisfies (\ref{PhiLipschitz}) and (\ref{PhiMonotone}). Suppose that 
\begin{equation} \label{bounds1}
||u||_{L^\infty(B_1)} \leq 1, \quad \int_{\mathbb{R}^n \setminus B_1} \frac{|u(y)|}{|y|^{n+2s}}dy \leq 1,
\end{equation}
and that for some $q \geq 2$, $\vartheta \in \mathbb{R}$ such that $0<(1+\vartheta q)/q<1$ and some $0<h_0<1$, we have
$$ \sup_{0<|h|<h_0} \left | \left | \frac{\delta_h^2 u}{ |h|^\frac{1+\vartheta q}{q}} \right | \right |_{L^{q}(B_{1})}^{q} < +\infty . $$
Then for any radius $4h_0<R \leq 1-2 h_0$, we have
$$ \sup_{0<|h|<h_0} \left | \left | \frac{\delta_h^2 u}{ |h|^\frac{1+2s +\vartheta q}{q+1}} \right | \right |_{L^{q+1}(B_{R-4h_0})}^{q+1} \leq C \left ( \sup_{0<|h|<h_0} \left | \left | \frac{\delta_h^2 u}{ |h|^\frac{1+\vartheta q}{q}} \right | \right |_{L^{q}(B_{R+4h_0})}^{q} + 1 \right ).$$
where $C=C(n,s,q,\lambda,h_0)>0$.
\end{prop}

\begin{proof}
\textbf{Step 1: Discrete differentiation of the equation.}
Set $r:=R-4h_0>0$ and fix some $h \in \mathbb{R}^n$ such that $0<|h|<h_0$. 
Let $\eta \in C_0^\infty(B_R)$ be a non-negative Lipschitz cutoff function satisfying
$$ \eta \equiv 1 \text{ in } B_r, \quad \eta \equiv 0 \text{ in } \mathbb{R}^n \setminus B_{(R+r)/2}, \quad |\nabla \eta| \leq \frac{C_1}{R-r}=\frac{C_1}{4h_0}.$$
Let us show that the function
$$ \varphi=J_{q+1} \left (\frac{\delta_h u}{|h|^\vartheta} \right ) \eta^2 = \left | \frac{\delta_h u}{|h|^\vartheta}\right |^{q-1}  \frac{\delta_h u}{|h|^\vartheta} \eta^2 $$
belongs to $W^{s,2}(B_R)$.
Since $||u||_{L^\infty(B_1)} \leq 1$ implies $||u||_{L^\infty(B_R)} \leq 1$ and also $||u_h||_{L^\infty(B_R)} \leq 1$, we have $||\varphi||_{L^\infty(B_R)} \leq \frac{2^q}{|h|^\vartheta}$ and therefore $\varphi \in L^\infty(B_R) \subset L^2(B_R)$. Moreover, note that the function $t \mapsto J_{q+1}(t)$ is Lipschitz continuous on the domain $t \in [-2,2]$ with Lipschitz constant $q2^{q-1}$. Therefore, since we have $||\delta_h u||_{L^\infty(B_R)} \leq 2$, we obtain
\begin{align*}
\int_{B_R} \int_{B_R} \frac{|J_{q+1}(\delta_h u(x))-J_{q+1}(\delta_h u(y))|^2}{|x-y|^{n+2s}}dydx & \leq C_2 \int_{B_R} \int_{B_R} \frac{|\delta_h u(x)-\delta_h u(y)|^2}{|x-y|^{n+2s}}dydx\\ & \leq 2C_2 ([u_h]_{W^{s,2}(B_R)}^2 + [u]_{W^{s,2}(B_R)}^2) < \infty,
\end{align*}
where $C_2=C_2(q)>0$, so that $J_{q+1}(\delta_h u) \in W^{s,2}(B_R)$. Thus, since the product of a function belonging to $W^{s,2}(B_R)$ and a Lipschitz function also belongs to $W^{s,2}(B_R)$ (cf. \cite[Lemma 5.3]{Hitch}), $\varphi = J_{q+1}(\delta_h u) \frac{\eta^2}{|h|^{\vartheta q}}$ also belongs to $W^{s,2}(B_R)$. \newline
Next, consider the function $\varphi_{-h}(x):=\varphi(x-h)$. Since both $\varphi$ and $\varphi_{-h}$ belong to $W^{s,2}(B_{R-h_0})$ and are compactly supported in $B_{R-h_0}$, in view of \cite[Lemma 2.11]{BLS} in particular both $\varphi$ and $\varphi_{-h}$ belong to $W^{s,2}_c(B_1)$, so that both $\varphi$ and $\varphi_{-h}$ are admissible test functions in (\ref{eq871}). Therefore, using $\varphi_{-h}$ as a test function in (\ref{eq871}) along with a change of variables yields
\begin{equation} \label{trest}
\begin{aligned}
0=& \int_{\mathbb{R}^n} \int_{\mathbb{R}^n} \frac{A(x,y)}{|x-y|^{n+2s}} \Phi(u(x)-u(y))(\varphi_{-h}(x)-\varphi_{-h}(y))dydx \\
= & \int_{\mathbb{R}^n} \int_{\mathbb{R}^n} \frac{A_h(x,y)}{|x-y|^{n+2s}} \Phi(u_h(x)-u_h(y))(\varphi(x)-\varphi(y))dydx ,
\end{aligned}
\end{equation}
where we have set $A_h(x,y):=A(x+h,y+h)$.
Moreover, testing (\ref{eq871}) with $\varphi$ yields
\begin{equation} \label{trest1}
\int_{\mathbb{R}^n} \int_{\mathbb{R}^n} \frac{A(x,y)}{|x-y|^{n+2s}} \Phi(u(x)-u(y))(\varphi(x)-\varphi(y))dydx=0.
\end{equation}
By subtracting (\ref{trest1}) from (\ref{trest}) and dividing by $0<|h|<h_0$,
we obtain
\begin{equation} \label{trest2}
\int_{\mathbb{R}^n} \int_{\mathbb{R}^n} \frac{A_h(x,y)\Phi(u_h(x)-u_h(y))-A(x,y)\Phi(u(x)-u(y))}{|h||x-y|^{n+2s}}(\varphi(x)-\varphi(y))dydx=0.
\end{equation}
Next, splitting the above integral and taking into account the choice of $\varphi$, we arrive at $$I_1+I_2+I_3=0,$$ where
\begin{align*}
I_1 :=\int_{B_R} \int_{B_R} & \frac{A_h(x,y)\Phi(u_h(x)-u_h(y))-A(x,y)\Phi(u(x)-u(y))}{|h|^{1+\vartheta q}|x-y|^{n+2s}} \\
& \times \left ( J_{q+1}(u_h(x)-u(x))\eta(x)^2 - J_{q+1}(u_h(y)-u(y))\eta(y)^2 \right ) dydx,
\end{align*}
\begin{align*}
I_2 :=\int_{B_{\frac{R+r}{2}}} \int_{\mathbb{R}^n \setminus B_R} & \frac{A_h(x,y)\Phi(u_h(x)-u_h(y))-A(x,y)\Phi(u(x)-u(y))}{|h|^{1+\vartheta q}|x-y|^{n+2s}} \\
& \times J_{q+1}(u_h(x)-u(x))\eta(x)^2 dydx,
\end{align*}
\begin{align*}
I_3 :=-\int_{\mathbb{R}^n \setminus B_R} \int_{B_{\frac{R+r}{2}}} & \frac{A_h(x,y)\Phi(u_h(x)-u_h(y))-A(x,y)\Phi(u(x)-u(y))}{|h|^{1+\vartheta q}|x-y|^{n+2s}} \\
& \times J_{q+1}(u_h(y)-u(y))\eta(y)^2 dydx,
\end{align*}
where we used that $\eta$ vanishes identically outside of $B_{(R+r)/2}$. \newline
\textbf{Step 2: Preliminary estimation of the local term $I_1$.}
Since $A \in \mathcal{L}_1(\lambda,B_1)$, we have $A(x,y)=a(x-y)$ for all $x,y \in B_1$ and some measurable function $a:\mathbb{R}^n \to \mathbb{R}$. Since for $x,y \in B_R$ we have $x+h,y+h \in B_1$, it follows that for all $x,y \in B_R$ we have $$ A_h(x,y)=A(x+h,y+h)=a((x+h)-(y+h)))=a(x-y)=A(x,y).$$
Therefore, we can rewrite $I_1$ as follows 
\begin{align*}
I_1 =\int_{B_R} \int_{B_R} & \frac{A(x,y)(\Phi(u_h(x)-u_h(y))-\Phi(u(x)-u(y)))}{|h|^{1+\vartheta q}|x-y|^{n+2s}} \\
& \times \left ( J_{q+1}(u_h(x)-u(x))\eta(x)^2 - J_{q+1}(u_h(y)-u(y))\eta(y)^2 \right ) dydx.
\end{align*}
Let us now concentrate on estimating $I_1$. First of all, we observe that 
\begin{align*}
& J_{q+1}(u_h(x)-u(x))\eta(x)^2-J_{q+1}(u_h(y)-u(y))\eta(y)^2 \\
= & \frac{(J_{q+1}(u_h(x)-u(x))-J_{q+1}(u_h(y)-u(y)))}{2} (\eta(x)^2+\eta(y)^2) \\
& + \frac{(J_{q+1}(u_h(x)-u(x))+J_{q+1}(u_h(y)-u(y)))}{2} (\eta(x)^2-\eta(y)^2).
\end{align*}
Therefore, we obtain
\begin{align*}
& (\Phi(u_h(x)-u_h(y))-\Phi(u(x)-u(y)) ) \left ( J_{q+1}(u_h(x)-u(x))\eta(x)^2-J_{q+1}(u_h(y)-u(y))\eta(y)^2 \right ) \\
\geq & (\Phi(u_h(x)-u_h(y))-\Phi(u(x)-u(y)) ) \\
& \times (J_{q+1}(u_h(x)-u(x))-J_{q+1}(u_h(y)-u(y))) \frac{(\eta(x)^2+\eta(y)^2)}{2} \\
& - |\Phi(u_h(x)-u_h(y))-\Phi(u(x)-u(y))| (|u_h(x)-u(x)|^q +|u_h(y)-u(y)|^q) \left | \frac{\eta(x)^2-\eta(y)^2}{2} \right |.
\end{align*}
Next, using the Lipschitz bound (\ref{PhiLipschitz}), Young's inequality and then Lemma \ref{elementary1}, for the negative term in the last display we deduce
\begin{align*}
& |\Phi(u_h(x)-u_h(y))-\Phi(u(x)-u(y))| (|u_h(x)-u(x)|^q +|u_h(y)-u(y)|^q) \left | \frac{\eta(x)^2-\eta(y)^2}{2} \right | \\
\leq & \frac{\lambda}{2} \text{ } |(u_h(x)-u_h(y))-(u(x)-u(y))| \\ & \times (|u_h(x)-u(x)|^{\frac{q-1}{2}}|u_h(x)-u(x)|^{\frac{q+1}{2}} +|u_h(y)-u(y)|^{\frac{q-1}{2}}|u_h(y)-u(y)|^{\frac{q+1}{2}}) \\ & \times (\eta(x)+\eta(y)) \left |\eta(x)-\eta(y) \right | \\
\leq & \frac{\lambda}{4\varepsilon} \left (|u_h(x)-u(x)|^{q+1} + |u_h(y)-u(y)|^{q+1} \right )\left |\eta(x)-\eta(y) \right |^2 \\
& + \frac{\lambda}{2} \varepsilon |(u_h(x)-u_h(y))-(u(x)-u(y))|^2 \left (|u_h(x)-u(x)|^{q-1} + |u_h(y)-u(y)|^{q-1} \right ) \\
& \times (\eta(x)^2+\eta(y)^2) \\
\leq & \frac{\lambda}{4\varepsilon} \left (|u_h(x)-u(x)|^{q+1} + |u_h(y)-u(y)|^{q+1} \right )\left |\eta(x)-\eta(y) \right |^2 \\
& + \lambda^2 \varepsilon (\Phi(u_h(x)-u_h(y))-\Phi(u(x)-u(y)))  \\ &\times (J_{q+1}(u_h(x)-u(x))-J_{q+1}(u_h(y)-u(y)))
(\eta(x)^2+\eta(y)^2),
\end{align*}
where $\varepsilon>0$ is arbitrary. By choosing $\varepsilon :=\frac{1}{2\lambda^2}$, combining the last two displays yields
\begin{align*}
I_1 \geq \frac{1}{4} \int_{B_R} \int_{B_R} & \frac{A(x,y)(\Phi(u_h(x)-u_h(y))-\Phi(u(x)-u(y)))}{|h|^{1+\vartheta q}|x-y|^{n+2s}} \\
& \times (J_{q+1}(u_h(x)-u(x))-J_{q+1}(u_h(y)-u(y)))
(\eta(x)^2+\eta(y)^2) dydx \\
- C_3 \int_{B_R}& \int_{B_R} \frac{A(x,y)\left(|u_h(x)-u(x)|^{q+1} + |u_h(y)-u(y)|^{q+1} \right )\left |\eta(x)-\eta(y) \right |^2}{|h|^{1+\vartheta q}|x-y|^{n+2s}} dydx,
\end{align*}
where $C_3=C_3(\lambda)>0$.
By using Lemma \ref{elementary2}, we can further estimate the first term of the previous display, which along with the bounds (\ref{eq1}) of $A$ leads to
\begin{equation} \label{I1bound}
\begin{aligned}
I_1 \geq c \int_{B_R} \int_{B_R} & \left | \frac{|\delta_h u(x)|^{\frac{q-1}{2}}\delta_h u(x)}{|h|^\frac{1+\vartheta q}{2}} - \frac{|\delta_h u(y)|^{\frac{q-1}{2}}\delta_h u(y)}{|h|^\frac{1+\vartheta q}{2}} \right |^2 
\frac{\eta(x)^2+\eta(y)^2}{|x-y|^{n+2s}} dydx \\
- C_4 \int_{B_R}& \int_{B_R} \frac{\left(|\delta_h u(x)|^{q+1} + |\delta_h u(y)|^{q+1} \right )\left |\eta(x)-\eta(y) \right |^2}{|h|^{1+\vartheta q}|x-y|^{n+2s}} dydx,
\end{aligned}
\end{equation}
where $c=c(\lambda,q)>0$ and $C_4=C_4(\lambda)>0$. Next, for simplicity we write
$$X:=\frac{|\delta_h u(x)|^{\frac{q-1}{2}}\delta_h u(x)}{|h|^\frac{1+\vartheta q}{2}} \quad \text{and} \quad Y:= \frac{|\delta_h u(y)|^{\frac{q-1}{2}}\delta_h u(y)}{|h|^\frac{1+\vartheta q}{2}}$$
and observe that by using the convexity of the function $t \mapsto t^2$, we obtain
\begin{align*}
|X \eta(x)-Y \eta(y)|^2 = & \left | (X-Y) \frac{\eta(x)+\eta(y)}{2}+(X+Y) \frac{\eta(x)-\eta(y)}{2} \right |^2 \\
\leq & \frac{1}{2} |X-Y|^2 |\eta(x)+\eta(y)|^2 + \frac{1}{2} |X+Y|^2 |\eta(x)-\eta(y)|^2 \\
\leq & |X-Y|^2 (\eta(x)^2 + \eta(y)^2) + (X^2+Y^2) |\eta(x)-\eta(y)|^2.
\end{align*}
Combining (\ref{I1bound}) with the last display yields
\begin{align*}
I_1 \geq & c \int_{B_R} \int_{B_R} \left | \frac{|\delta_h u(x)|^{\frac{q-1}{2}}\delta_h u(x)}{|h|^\frac{1+\vartheta q}{2}} \eta(x) - \frac{|\delta_h u(y)|^{\frac{q-1}{2}}\delta_h u(y)}{|h|^\frac{1+\vartheta q}{2}} \eta(y) \right |^2 \frac{1}{|x-y|^{n+2s}}dydx \\
& - c \int_{B_R} \int_{B_R} \left (\frac{|\delta_h u(x)|^{q+1}}{|h|^{1+\vartheta q}} + \frac{|\delta_h u(y)|^{q+1}}{|h|^{1+\vartheta q}} \right ) \frac{|\eta(x)-\eta(y)|^2}{|x-y|^{n+2s}} dydx \\ 
& - C_4 \int_{B_R} \int_{B_R} \frac{\left(|\delta_h u(x)|^{q+1} + |\delta_h u(y)|^{q+1} \right )\left |\eta(x)-\eta(y) \right |^2}{|h|^{1+\vartheta q}|x-y|^{n+2s}} dydx \\
= & c \left [\frac{|\delta_h u|^{\frac{q-1}{2}}\delta_h u}{|h|^\frac{1+\vartheta q}{2}} \eta \right ]_{W^{s,2}(B_R)}^2 \\
& - C_5 \int_{B_R} \int_{B_R} \left (\frac{|\delta_h u(x)|^{q+1}}{|h|^{1+\vartheta q}} + \frac{|\delta_h u(y)|^{q+1}}{|h|^{1+\vartheta q}} \right ) \frac{|\eta(x)-\eta(y)|^2}{|x-y|^{n+2s}} dydx,
\end{align*}
where $C_5=C_5(\lambda,q)>0$. By combining the above estimate for $I_1$ with the identity $I_1+I_2+I_3=0$, we arrive at
\begin{equation} \label{firstest}
\left [\frac{|\delta_h u|^{\frac{q-1}{2}}\delta_h u}{|h|^\frac{1+\vartheta q}{2}} \eta \right ]_{W^{s,2}(B_R)}^2 \leq C_6 (I_{1,1} + |I_2|+|I_3|),
\end{equation}
where $C_6=C_6(\lambda,q)>0$ and
$$ I_{1,1}:= \int_{B_R} \int_{B_R} \left (\frac{|\delta_h u(x)|^{q+1}}{|h|^{1+\vartheta q}} + \frac{|\delta_h u(y)|^{q+1}}{|h|^{1+\vartheta q}} \right ) \frac{|\eta(x)-\eta(y)|^2}{|x-y|^{n+2s}} dydx.$$
Our next goal is to estimate the terms $I_{1,1}$, $|I_2|$ and $|I_3|$. \newline
\textbf{Step 3: Estimating the local term $I_{1,1}$.}
In order to estimate $I_{1,1}$, observe that for any $x \in B_R$ changing variables and integrating in polar coordinates yields
\begin{equation} \label{intpolar}
\int_{B_{R}} \frac{dy}{|x-y|^{n+2s-2}} \leq \int_{B_{2R}} \frac{dz}{|z|^{n+2s-2}} = C_7 R^{2-2s} \leq C_7,
\end{equation}
where $C_7=C_7(n,s) >0$.
Since by construction $\eta$ is Lipschitz with Lipschitz constant $\frac{C_1}{4h_0}$, along with (\ref{intpolar}) we obtain
\begin{align*}
\int_{B_R} \int_{B_R} \frac{|\delta_h u(x)|^{q+1}}{|h|^{1+\vartheta q}} \frac{|\eta(x)-\eta(y)|^2}{|x-y|^{n+2s}} dydx \leq & \left(\frac{C_1}{4h_0}\right )^2 \int_{B_R} \left (\int_{B_{R}} \frac{dy}{|x-y|^{n+2s-2}} \right)\frac{|\delta_h u(x)|^{q+1}}{|h|^{1+\vartheta q}} dx \\
\leq & C_7 \left(\frac{C_1}{4h_0}\right )^2 \int_{B_R} \frac{|\delta_h u(x)|^{q+1}}{|h|^{1+\vartheta q}} dx \\
\leq & C_8 ||u||_{L^\infty(B_{R+h_0})} \int_{B_R} \frac{|\delta_h u(x)|^{q}}{|h|^{1+\vartheta q}} dx \\ & \leq C_8 \int_{B_R} \frac{|\delta_h u(x)|^{q}}{|h|^{1+\vartheta q}} dx ,
\end{align*}
where we used that $R+h_0 \leq 1$ and $||u||_{L^\infty(B_{1})}\leq 1$ in order to obtain the last inequality and $C_8=C_8(n,s,q,\lambda,h_0)>0$. In the same way we have 
\begin{align*}
\int_{B_R} \int_{B_R} \frac{|\delta_h u(y)|^{q+1}}{|h|^{1+\vartheta q}} \frac{|\eta(x)-\eta(y)|^2}{|x-y|^{n+2s}} dydx \leq C_8 \int_{B_R} \frac{|\delta_h u(y)|^{q}}{|h|^{1+\vartheta q}} dy ,
\end{align*}
so that we obtain
\begin{equation} \label{I11}
I_{1,1} \leq 2C_8 \int_{B_R} \frac{|\delta_h u(x)|^{q}}{|h|^{1+\vartheta q}} dx. 
\end{equation}
\textbf{Step 4: Estimating the nonlocal terms $I_2$ and $I_3$.}
Next, let us estimate the nonlocal terms $I_2$ and $I_3$, which can be treated in the same way. Since $||u||_{L^\infty(B_{1})}\leq 1$ and $(R+r)/2+h_0 \leq 1$, by additionally using the bound (\ref{PhiLipschitz}) of $\Phi$ with $t=u_h(x)-u_h(y)$ and $t^\prime=0$, for almost every $x \in B_{(R-r)/2}$ and any $y \in \mathbb{R}^n \setminus B_R$ we have
\begin{align*}
\left |\Phi(u_h(x)-u_h(y)) J_{q+1}(\delta_h u(x)) \right | \leq & \lambda \left(||u||_{L^\infty(B_{(R+r)/2+h_0})}+|u_h(y)| \right) |\delta_h u(x)|^q \\
\leq & \lambda \left(1+|u_h(y)| \right) |\delta_h u(x)|^q
\end{align*}
and similarly
\begin{align*}
\left |\Phi(u(x)-u(y)) J_{q+1}(\delta_h u(x)) \right |
\leq \lambda \left(1+|u(y)| \right) |\delta_h u(x)|^q.
\end{align*} 
By using the upper bound in (\ref{eq1}) of $A$ (which trivially also holds for $A_h$) and the fact that $0 \leq \eta \leq 1$ and then the last two displays, we deduce
\begin{equation} \label{critest}
\begin{aligned}
|I_2| \leq & \lambda \int_{B_{\frac{R+r}{2}}} \int_{\mathbb{R}^n \setminus B_R} \frac{\left (\left |\Phi(u_h(x)-u_h(y)) \right |+\left |\Phi(u(x)-u(y)) \right | \right ) \left |J_{q+1}(\delta_h u(x)) \right |}{|h|^{1+\vartheta q}|x-y|^{n+2s}} dydx \\
\leq & 2 \lambda^2 \int_{B_{\frac{R+r}{2}}} \int_{\mathbb{R}^n \setminus B_R} \frac{\left(1+|u_h(y)| +|u(y)| \right) |\delta_h u(x)|^q}{|h|^{1+\vartheta q}|x-y|^{n+2s}} dydx.
\end{aligned}
\end{equation}
For any $x \in B_{(R+r)/2)}$, we have $B_{(R-r)/2}(x) \subset B_R$, which in view of integration in polar coordinates along with the fact that $R-r=4h_0$ leads to
\begin{align*}
\int_{\mathbb{R}^n \setminus B_R} \frac{dy}{|x-y|^{n+2s}} \leq \int_{\mathbb{R}^n \setminus B_{\frac{R-r}{2}}(x)} \frac{dy}{|x-y|^{n+2s}} = \int_{\mathbb{R}^n \setminus B_{\frac{R-r}{2}}} \frac{dz}{|z|^{n+2s}}
 = C_9 \left (\frac{R-r}{2} \right )^{-2s}=C_{10},
\end{align*}
where $C_9=C_9(n,s)>0$ and $C_{10}= C_9(2h_0)^{-2s}$.
Using Lemma \ref{tailestz}, the change of variables $z=y+h$ and then Lemma \ref{tail}, for any $x \in B_{(R+r)/2}$ we obtain
\begin{align*}
\int_{\mathbb{R}^n \setminus B_R} \frac{|u_h(y)|}{|x-y|^{n+2s}}dy \leq &
\left (\frac{2R}{R-r} \right)^{n+2s} \int_{\mathbb{R}^n \setminus B_R} \frac{|u_h(y)|}{|y|^{n+2s}}dy \\ \leq & (2h_0)^{-(n+2s)} \int_{\mathbb{R}^n \setminus B_R(h)} \frac{|u(z)|}{|h-z|^{n+2s}}dz \\
\leq & (2h_0R)^{-(n+2s)} \left (||u||_{L^1(B_1)} + \int_{\mathbb{R}^n \setminus B_{1}} \frac{|u(z)|}{|z|^{n+2s}}dz \right) \\
\leq & (8h_0^2)^{-(n+2s)} \left (||u||_{L^\infty(B_1)}|B_1| + \int_{\mathbb{R}^n \setminus B_{1}} \frac{|u(z)|}{|z|^{n+2s}}dz \right ) \leq C_{11},
\end{align*}
where $C_{11}=C_{11}(n,s,h_0)>0$. Here we also used the the fact that $R>4h_0$ and the bounds imposed on $u$.
The term involving $u$ can be estimated similarly.
In fact, by using Lemma \ref{tailestz} and Lemma \ref{tail}, for any $x \in B_{(R+r)/2}$ we obtain
\begin{align*}
\int_{\mathbb{R}^n \setminus B_R} \frac{|u(y)|}{|x-y|^{n+2s}}dy
\leq & (2h_0)^{-(n+2s)} \int_{\mathbb{R}^n \setminus B_R} \frac{|u(y)|}{|y|^{n+2s}}dy \\
\leq & (8h_0^2)^{-(n+2s)} \left (||u||_{L^\infty(B_1)}|B_1| + \int_{\mathbb{R}^n \setminus B_{1}} \frac{|u(y)|}{|y|^{n+2s}}dy \right )
\leq C_{12},
\end{align*}
where $C_{12}=C_{12}(n,s,h_0)>0$.
By combining the above estimates with (\ref{critest}) and the observation that $|I_3|$ can be estimated in the same way, we arrive at
$$
|I_2|+|I_3| \leq C_{13} \int_{B_{\frac{R+r}{2}}} \frac{|\delta_h u(x)|^q}{|h|^{1+\vartheta q}} dx \leq C_{13} \int_{B_{R}} \frac{|\delta_h u(x)|^q}{|h|^{1+\vartheta q}} dx,
$$
where $C_{13}=C_{13}(n,s,\lambda,h_0)>0$. By combining this estimate with (\ref{I11}) and (\ref{firstest}), we find the estimate
\begin{equation} \label{impest}
\left [\frac{|\delta_h u|^{\frac{q-1}{2}}\delta_h u}{|h|^\frac{1+\vartheta q}{2}} \eta \right ]_{W^{s,2}(B_R)}^2 \leq C_{14} \int_{B_{R}} \frac{|\delta_h u(x)|^q}{|h|^{1+\vartheta q}} dx,
\end{equation}
where $C_{14}=C_{14}(n,s,q,\lambda,h_0)>0$. \newline
\textbf{Step 5: Conclusion.}
Let $\xi \in \mathbb{R}^n \setminus \{0\}$ to be chosen such that $|\xi|<h_0$. Applying Lemma \ref{elementary0} with 
$$ X=u(x+h+\xi)-u(x+\xi), \quad Y= u(x+h)-u(x), \quad p=\frac{q+1}{2}$$
leads to
$$ \left | \left | \frac{\delta_\xi \delta_h u}{|\xi|^\frac{2s}{q+1} |h|^\frac{1+\vartheta q}{q+1}} \right | \right |_{L^{q+1}(B_r)}^{q+1} \leq C_{15} \left | \left | \frac{\delta_\xi \left ( |\delta_h u|^\frac{q-1}{2} \delta_h u \right )}{|\xi|^s |h|^\frac{1+\vartheta q}{2}} \right | \right |_{L^{2}(B_r)}^{2} \leq C_{15} \left | \left |\eta \frac{\delta_\xi}{|\xi|^s} \left ( \frac{ |\delta_h u|^\frac{q-1}{2} \delta_h u }{ |h|^\frac{1+\vartheta q}{2}}\right ) \right | \right |_{L^{2}(\mathbb{R}^n)}^{2}, $$
where $C_{15}=C_{15}(q)>0$. Here we also used that $\eta \equiv 1$ in $B_r$ in order to obtain the last inequality. Next, we observe that by the discrete Leibniz rule (cf. \cite[Formula (2.1)]{BLS}), we can write
$$ \eta \delta_\xi \left ( |\delta_h u|^\frac{q-1}{2} \delta_h u \right ) = \delta_\xi \left (\eta |\delta_h u|^\frac{q-1}{2} \delta_h u \right ) - \left ( |\delta_h u|^\frac{q-1}{2} \delta_h u \right )_{\xi} \delta_\xi \eta.$$
We arrive at
\begin{equation} \label{diff1}
\begin{aligned}
& \left | \left | \frac{\delta_\xi \delta_h u}{|\xi|^\frac{2s}{q+1} |h|^\frac{1+\vartheta q}{q+1}} \right | \right |_{L^{q+1}(B_r)}^{q+1} \\ \leq & C_{16} \left | \left | \frac{\delta_\xi}{|\xi|^s} \left ( \frac{ |\delta_h u|^\frac{q-1}{2} (\delta_h u) \eta}{ |h|^\frac{1+\vartheta q}{2}}\right ) \right | \right |_{L^{2}(\mathbb{R}^n)}^{2} +  C_{16} \left | \left | \frac{\delta_\xi \eta}{|\xi|^s} \frac{\left ( |\delta_h u|^\frac{q-1}{2} \delta_h u \right )_\xi}{ |h|^\frac{1+\vartheta q}{2}} \right | \right |_{L^{2}(\mathbb{R}^n)}^{2}, 
\end{aligned}
\end{equation}
where $C_{16}=2C_{15}$. By applying the first part of Proposition \ref{diffsobolev} with 
$$\psi=\frac{ |\delta_h u|^\frac{q-1}{2} (\delta_h u) \eta}{ |h|^\frac{1+\vartheta q}{2}},$$
for the first term on the right-hand side of (\ref{diff1}) we obtain
\begin{align*}
\sup_{|\xi|>0} \left | \left | \frac{\delta_\xi}{|\xi|^s} \left ( \frac{ |\delta_h u|^\frac{q-1}{2} (\delta_h u) \eta}{ |h|^\frac{1+\vartheta q}{2}}\right ) \right | \right |_{L^{2}(\mathbb{R}^n)}^{2} \leq & C_{17}\left (\frac{R}{r} \right)^n \left(\frac{R}{R-r} \right )^{3} \left [\frac{|\delta_h u|^{\frac{q-1}{2}}\delta_h u}{|h|^\frac{1+\vartheta q}{2}} \eta \right ]_{W^{s,2}(B_R)}^2 \\
\leq & C_{18} \left [\frac{|\delta_h u|^{\frac{q-1}{2}}\delta_h u}{|h|^\frac{1+\vartheta q}{2}} \eta \right ]_{W^{s,2}(B_R)}^2,
\end{align*}
where $C_{17}=C_{17}(n,s)>0$ and $C_{18}=C_{18}(n,s,h_0)>0$. By using that $\eta$ is Lipschitz and that $\xi < h_0$, along with the assumption that $||u||_{L^\infty(B_1)} \leq 1$ we estimate the second term on the right-hand side of (\ref{diff1}) as follows
\begin{align*}
\left | \left | \frac{\delta_\xi \eta}{|\xi|^s} \frac{\left ( |\delta_h u|^\frac{q-1}{2} \delta_h u \right )_\xi}{ |h|^\frac{1+\vartheta q}{2}} \right | \right |_{L^{2}(\mathbb{R}^n)}^{2} \leq & C_{19} \left | \left | \frac{\left ( |\delta_h u|^\frac{q-1}{2} \delta_h u \right )_\xi}{ |h|^\frac{1+\vartheta q}{2}} \right | \right |_{L^{2}(B_{\frac{R+r}{2}+h_0})}^{2} \\
\leq & C_{19} \int_{B_{\frac{R+r}{2}+2h_0}} 
\frac{|\delta_h u(x)|^{q+1}}{|h|^{1+\vartheta q}} dx \leq C_{19} \int_{B_{R}} \frac{|\delta_h u(x)|^q}{|h|^{1+\vartheta q}} dx,
\end{align*}
where $C_{19}=C_{19}(n,h_0)>0$. Therefore, we arrive at
$$ \left | \left | \frac{\delta_h^2 u}{|\xi|^\frac{2s}{q+1} |h|^\frac{1+\vartheta q}{q+1}} \right | \right |_{L^{q+1}(B_r)}^{q+1} \leq C_{20} \left [\frac{|\delta_h u|^{\frac{q-1}{2}}\delta_h u}{|h|^\frac{1+\vartheta q}{2}} \eta \right ]_{W^{s,2}(B_R)}^2 +C_{20}  \int_{B_{R}} \frac{|\delta_h u(x)|^q}{|h|^{1+\vartheta q}} dx,$$
where $C_{20}=C_{20}(n,s,q,h_0)>0$. We now choose $\xi=h$ and take the supremum over $h$ for $0<|h|<h_0$, so that together with (\ref{impest}) we obtain
\begin{equation} \label{almostdone}
\sup_{0<|h|<h_0} \left | \left | \frac{\delta_h^2 u}{ |h|^\frac{1+2s +\vartheta q}{q+1}} \right | \right |_{L^{q+1}(B_r)}^{q+1} \leq C_{21} \sup_{0<|h|<h_0} \left | \left | \frac{\delta_h u}{ |h|^\frac{1+\vartheta q}{q}} \right | \right |_{L^{q}(B_R)}^{q},
\end{equation}
where $C_{21}=C_{21}(n,s,q,h_0,\lambda)>0$. Next, we use the fact that by \cite[Lemma 2.6]{BLS} applied with $\beta=(1+\vartheta q)/q<1$, on the right-hand side of (\ref{almostdone}) we can replace the first-order difference quotient by a corresponding second-order difference quotient in the following way
$$ \sup_{0<|h|<h_0} \left | \left | \frac{\delta_h u}{ |h|^\frac{1+\vartheta q}{q}} \right | \right |_{L^{q}(B_R)}^{q} \leq C_{22} \left ( \sup_{0<|h|<h_0} \left | \left | \frac{\delta_h^2 u}{ |h|^\frac{1+\vartheta q}{q}} \right | \right |_{L^{q}(B_{R+h_0})}^{q} + ||u||_{L^q(B_{R+h_0})}^q \right ) ,$$
where $C_{22}=C_{22}(n,q,\vartheta,h_0)>0$.
By combining the last display with (\ref{almostdone}) and using that $||u||_{L^q(B_{R+h_0})}^q \leq ||u||_{L^\infty(B_1)}^q |B_{1}| \leq |B_{1}|$, we conclude that
$$\sup_{0<|h|<h_0} \left | \left | \frac{\delta_h^2 u}{ |h|^\frac{1+2s +\vartheta q}{q+1}} \right | \right |_{L^{q+1}(B_r)}^{q+1} \leq C \left ( \sup_{0<|h|<h_0} \left | \left | \frac{\delta_h^2 u}{ |h|^\frac{1+\vartheta q}{q}} \right | \right |_{L^{q}(B_{R+4h_0})}^{q} + 1 \right ) ,$$
where $C=C(n,s,q,\vartheta,h_0,\lambda)>0$. Since $r=R-4h_0$, the proof is finished.
\end{proof}
\subsection{An iteration argument}
We now use an iteration argument based on Proposition \ref{increment} in order to obtain the following higher H\"older regularity result.
\begin{thm} \label{modC2sreg}
	Let $R>0$, $x_0 \in \mathbb{R}^n$ and $\lambda \geq 1$. Consider a kernel coefficient $A \in \mathcal{L}_1(\lambda,B_R(x_0))$, suppose that $\Phi$ satisfies (\ref{PhiLipschitz}) and (\ref{PhiMonotone}) with respect to $\lambda$ and assume that $u \in W^{s,2}(B_{R}(x_0)) \cap L^1_{2s}(\mathbb{R}^n) \cap L^\infty(B_R(x_0))$ is a local weak solution of the equation $L_{A}^\Phi u = 0$ in $B_{R}(x_0)$. Then for any $0<\alpha<\min \left \{2s,1 \right\}$, we have
	\begin{equation} \label{C2sreg5}
	[u]_{C^{\alpha}(B_{R/2}(x_0))} \leq \frac{C}{R^\alpha} \left (||u||_{L^\infty(B_R(x_0))}+ R^{s-\frac{n}{2}} [u]_{W^{s,2}(B_{R}(x_0))} + R^{2s} \int_{\mathbb{R}^n \setminus B_{R}(x_0)} \frac{|u(y)|}{|x_0-y|^{n+2s}}dy \right ),
	\end{equation}
	where $C=C(n,s,\lambda,\alpha)>0$.
\end{thm}
\begin{proof}
If $u \equiv 0$ a.e., then the assertion is trivially satisfied. Otherwise, set
$$ M_{R,x_0}:= ||u||_{L^\infty(B_R(x_0))} + R^{s-\frac{n}{2}} [u]_{W^{s,2}(B_R(x_0))} + R^{2s} \int_{\mathbb{R}^n \setminus B_{R}(x_0)} \frac{|u(y)|}{|x_0-y|^{n+2s}}dy>0. $$ Consider the scaled function
$$ u_1(x):=\frac{1}{M_{R,x_0}}u(Rx+x_0) $$
and also 
$$A_1(x,y):= A(Rx+x_0,Ry+x_0), \quad \Phi_1(t):= \frac{1}{M_{R,x_0}} \Phi(M_{R,x_0} t).$$
Observe that $u_1$ belongs to $W^{s,2}(B_{1}) \cap L^1_{2s}(\mathbb{R}^n) \cap L^\infty(B_1)$ and is a weak solution of $L_{A_1}^{\Phi_1} u_1 = 0$ in $B_1$. Moreover, it is easy to verify that $A_1 \in \mathcal{L}_1(\lambda,B_1)$ and that $\Phi_1$ satisfies (\ref{PhiLipschitz}) and (\ref{PhiMonotone}) with respect to $\lambda$. Furthermore, by using changes of variables it is straightforward to verify that $u_1$ satisfies
\begin{equation} \label{smallu1}
||u_1||_{L^\infty(B_1)} \leq 1, \quad \int_{\mathbb{R}^n \setminus B_1} \frac{|u_1(y)|}{|y|^{n+2s}}dy \leq 1, \quad [u_1]_{W^{s,2}(B_1)} \leq 1.
\end{equation}
Therefore, the conclusion of Proposition \ref{increment} is valid with respect to $u_1$.
For $i \in \mathbb{N}_0$, we define the sequences $$q_i:=2+i, \quad \vartheta_i:= \frac{2si+2s-1}{2+i}. $$
In particular, we have 
\begin{equation} \label{Limit5}
\lim_{i \to \infty} q_i= \infty, \quad \lim_{i \to \infty} \vartheta_i = 2s.
\end{equation}
We split the further proof into two cases. \newline
\textbf{Case 1: $s \leq 1/2$.} Fix $0<\alpha<2s$. In view of (\ref{Limit5}), we can find some large enough $i_\infty \in \mathbb{N}$ such that 
\begin{equation} \label{applxy}
\alpha < \frac{1}{q_{i_\infty}} + \vartheta_{i_\infty} - \frac{n}{q_{i_\infty}}. 
\end{equation}
For $i=0,..., i_\infty$, define
$$ h_0:=\frac{1}{64 i_\infty}, \quad R_i:=\frac{7}{8} - 4(2i+1)h_0= \frac{7}{8} - \frac{2i+1}{16 i_\infty}. $$
We note that
\begin{equation} \label{relations5}
R_0+4h_0=\frac{7}{8}, \quad R_{i_\infty-1}-4h_0=\frac{3}{4}, \quad R_i-4h_0=R_{i+1}+4h_0 \text{ } (i=0,...,i_\infty-2).
\end{equation}
Since $s \leq 1/2$, for $i=0,...,i_\infty-1$ we have $0<(1+\vartheta_i q_i)/q_i<1$. Therefore, for $i=0,...,i_\infty-1$ we can apply Proposition \ref{increment} to 
$$ R=R_i, \quad \vartheta=\vartheta_i, \quad q=q_i,$$
so that along with (\ref{relations5}) and the observation that by construction
$$ \frac{1+2s+\vartheta_i q_i}{q_i+1} = \frac{1+\vartheta_{i+1} q_{i+1}}{q_{i+1}}, $$
we obtain the following estimates
$$ \sup_{0<|h|<h_0} \left | \left | \frac{\delta_h^2 u_1}{ |h|^\frac{1+\vartheta_{1} q_{1}}{q_{1}}} \right | \right |_{L^{q_1}(B_{R_1+4h_0})} \leq C_0 \left ( \sup_{0<|h|<h_0} \left | \left | \frac{\delta_h^2 u_1}{ |h|^s} \right | \right |_{L^{2}(B_{7/8})} + 1 \right ),$$
$$ \sup_{0<|h|<h_0} \left | \left | \frac{\delta_h^2 u_1}{ |h|^\frac{1+\vartheta_{i+1} q_{i+1}}{q_{i+1}}} \right | \right |_{L^{q_{i+1}}(B_{R_{i+1}+4h_0})} \leq C_0 \left ( \sup_{0<|h|<h_0} \left | \left | \frac{\delta_h^2 u_1}{ |h|^\frac{1+\vartheta_{i} q_{i}}{q_{i}}} \right | \right |_{L^{q_i}(B_{R_i+4h_0})} + 1 \right ), \text{ } i=1,...,i_\infty-2,$$
and
$$ \sup_{0<|h|<h_0} \left | \left | \frac{\delta_h^2 u_1}{ |h|^{\frac{1}{q_{i_\infty}}+\vartheta_{i_\infty}}} \right | \right |_{L^{q_{i_\infty}}(B_{3/4})} \leq C_0 \left ( \sup_{0<|h|<h_0} \left | \left | \frac{\delta_h^2 u_1}{ |h|^\frac{1+\vartheta_{i_\infty-1} q_{i_\infty-1}}{q_{i_\infty-1}}} \right | \right |_{L^{q_{i_\infty-1}}(B_{R_{i_\infty-1}+4h_0})} + 1 \right ),$$
where $C_0=C_0(n,s,\lambda,\alpha)$. Combining the above estimates leads to the estimate
\begin{equation} \label{impeq5}
 \sup_{0<|h|<h_0} \left | \left | \frac{\delta_h^2 u_1}{ |h|^{\frac{1}{q_{i_\infty}}+\vartheta_{i_\infty}}} \right | \right |_{L^{q_{i_\infty}}(B_{3/4})} \leq C_1 \left ( \sup_{0<|h|<h_0} \left | \left | \frac{\delta_h^2 u_1}{ |h|^s} \right | \right |_{L^{2}(B_{7/8})} + 1 \right ),
\end{equation}
where $C_1=C_1(n,s,\lambda,\alpha)>0$. By taking into account the relation 
$$\delta_h u_1 = \frac{1}{2} (\delta_{2h} u_1 - \delta_h^2 u_1)$$
and then using the second part of Proposition \ref{diffsobolev} and then (\ref{smallu1}), we deduce
\begin{equation} \label{propappl}
\begin{aligned}
\sup_{0<|h|<h_0} \left | \left | \frac{\delta_h^2 u_1}{ |h|^s} \right | \right |_{L^{2}(B_{7/8})} \leq & 2 \sup_{0<|h|<h_0} \left | \left | \frac{\delta_h u_1}{ |h|^s} \right | \right |_{L^{2}(B_{7/8})} \\
\leq & C_2 \left ( [u_1]_{W^{s,2}(B_{7/8+2h_0})} + [u_1]_{L^\infty(B_{7/8+2h_0})} \right ) \\
\leq & C_2 \left ( [u_1]_{W^{s,2}(B_{1})} + [u_1]_{L^\infty(B_{1})} \right ) \leq C_2(n,s,\alpha) .
\end{aligned}
\end{equation}
By combining (\ref{impeq5}) with (\ref{propappl}) and setting $$\beta:=\frac{1}{q_{i_\infty}}+\vartheta_{i_\infty} \in (0,1),$$ we arrive at
\begin{equation} \label{univbound5}
\sup_{0<|h|<h_0} \left | \left | \frac{\delta_h^2 u_1}{ |h|^\beta} \right | \right |_{L^{q_{i_\infty}}(B_{3/4})} \leq C_3(n,s,\lambda,\alpha).
\end{equation}
In order to proceed, we fix a cutoff function $\chi \in C_0^\infty(B_{5/8})$ with the properties
$$ 0 \leq \chi \leq 1, \quad \chi \equiv 1 \text{ in } B_{1/2}, \quad |\nabla \chi| \leq C_4, \quad |\nabla^2 \chi| \leq C_4,$$
where by $\nabla^2 \chi$ we denote the Hessian of $\chi$ and $C_4=C_4(n)>0$.
In particular, since $0<\beta<1$, for any $h \in \mathbb{R}^n$ with $|h|>0$ we have
$$ \frac{|\delta_h \chi|}{ |h|^\beta} \leq C_5, \quad \frac{|\delta_h^2 \chi|}{ |h|^\beta} \leq C_5,$$
where $C_5=C_5(n)>0$. Together with the identity
\begin{equation} \label{2ndproductrule}
\delta^2_h(u_1 \chi)) = \chi_{2h} \delta_h^2u_1 + 2 \delta_h u_1 \delta_h \chi_h + u_1 \delta_h^2 \chi,
\end{equation}
(\ref{univbound5}) and (\ref{smallu1}), for $0<|h|<h_0$ we obtain 
\begin{align*}
& \left | \left | \frac{\delta_h^2 (u_1 \chi)}{ |h|^\beta} \right | \right |_{L^{q_{i_\infty}}(\mathbb{R}^n)} \\
\leq & 2 \left ( \left | \left | \frac{\chi_{2h} \delta_h^2u_1}{ |h|^\beta} \right | \right |_{L^{q_{i_\infty}}(\mathbb{R}^n)} + \left | \left | \frac{\delta_h u_1 \delta_h \chi_h}{ |h|^\beta} \right | \right |_{L^{q_{i_\infty}}(\mathbb{R}^n)} + \left | \left | \frac{u_1 \delta_h^2 \chi}{ |h|^\beta} \right | \right |_{L^{q_{i_\infty}}(\mathbb{R}^n)} \right ) \\
\leq & 2 \left ( \left | \left | \frac{\delta_h^2u_1}{ |h|^\beta} \right | \right |_{L^{q_{i_\infty}}(B_{5/8+2h_0})} + \left | \left |\delta_h u_1 \right | \right |_{L^{q_{i_\infty}}(B_{5/8+2h_0})} + \left | \left | u_1 \right | \right |_{L^{q_{i_\infty}}(B_{5/8+2h_0})} \right ) \\
\leq & C_6 \left ( \left | \left | \frac{\delta_h^2u_1}{ |h|^\beta} \right | \right |_{L^{q_{i_\infty}}(B_{3/4})} + \left | \left | u_1 \right | \right |_{L^{\infty}(B_{3/4})} \right ) \leq C_7(n,s,\lambda,\alpha).
\end{align*}
Since moreover by (\ref{smallu1}), for $|h| \geq h_0$ we have
$$ \left | \left | \frac{\delta_h^2 (u_1 \chi)}{ |h|^\beta} \right | \right |_{L^{q_{i_\infty}}(\mathbb{R}^n)} \leq C_8 \left | \left | u_1 \right | \right |_{L^{\infty}(B_{3/4})} \leq C_8(n,s,\alpha).$$
by Lemma \ref{embedding5} it follows that
\begin{equation} \label{Jl}
[u_1 \chi]_{\mathcal{N}_\infty^{\beta,q_{i_\infty}}(\mathbb{R}^n)} \leq C_9 [u_1 \chi]_{\mathcal{B}_\infty^{\beta,q_{i_\infty}}(\mathbb{R}^n)} = C_9 \sup_{h>0} \left | \left | \frac{\delta_h^2 (u_1 \chi)}{ |h|^\beta} \right | \right |_{L^{q_{i_\infty}}(\mathbb{R}^n)} \leq C_{10}(n,s,\lambda,\alpha).
\end{equation}
Along with Lemma \ref{Holderemb} with our choice of $\beta$ and $q=q_{i_\infty}$ (which is applicable in view of (\ref{applxy})), we obtain 
\begin{equation} \label{J1}
\begin{aligned}
[u_1]_{C^\alpha(B_{1/2})} = [u_1 \chi]_{C^\alpha(B_{1/2})} \leq & C_{11} \left ([u_1 \chi]_{\mathcal{N}_\infty^{\beta,q_{i_\infty}}(\mathbb{R}^n)} \right)^\frac{\alpha q_{i_\infty} +n}{\beta q_{i_\infty}} \left (||u_1 \chi||_{L^{q_{i_\infty}}(\mathbb{R}^n)} \right)^{1-\frac{\alpha q_{i_\infty} +n}{\beta q_{i_\infty}}} \\
\leq & C_{12} \left (||u_1||_{L^{\infty}(B_{5/8})} \right)^{1-\frac{\alpha q_{i_\infty} +n}{\beta q_{i_\infty}}} \leq C(n,s,\lambda,\alpha).
\end{aligned}
\end{equation}
Finally, rescaling yields the desired estimate, namely (\ref{C2sreg5}). This finishes the proof in the case when $s \leq 1/2$. \newline
\textbf{Case 2: $s > 1/2$.} Fix $0<\alpha<1$.
Since in view of (\ref{Limit5}) we have $$\lim_{i \to \infty} \frac{1+\vartheta_{i} q_{i}}{q_{i}}=2s>1$$
and the expression $\frac{1+\vartheta_{i} q_{i}}{q_{i}}$ is increasing in $i$, there exists some $i_\infty \in \mathbb{N}$ such that
\begin{equation} \label{choice4}
\frac{1+\vartheta_{i} q_{i}}{q_{i}}<1 \text{ for any } i=0,...,i_{\infty}-1 \text{ and } \frac{1+\vartheta_{i_{\infty}} q_{i_{\infty}}}{q_{i_{\infty}}}\geq 1.
\end{equation}
Next, we choose $j_\infty \in \mathbb{N}$ large enough such that
$$
\alpha<1-\frac{n}{i_\infty+j_\infty}.
$$
Moreover, we choose some $\varepsilon \in (0,1)$ such that
\begin{equation} \label{choice7}
\alpha < 1-\varepsilon-\frac{n}{i_\infty+j_\infty}
\end{equation}
and let $\gamma:=1-\varepsilon$. Furthermore, similar to the previous case, for $i=0,..., i_\infty+j_\infty$ we define
$$ h_0:=\frac{1}{64 (i_\infty + j_\infty)}, \quad R_i:=\frac{7}{8} - 4(2i+1)h_0= \frac{7}{8} - \frac{2i+1}{16 (i_\infty + j_\infty)} $$
and note that
\begin{equation} \label{relations52}
R_0+4h_0=\frac{7}{8}, \quad R_{i_\infty+ j_\infty-1}-4h_0=\frac{3}{4}, \quad R_i-4h_0=R_{i+1}+4h_0 \text{ } (i=0,...,i_\infty+ j_\infty-2).
\end{equation}
In view of (\ref{choice4}), for $i=0,...,i_\infty-1$ we can apply Proposition \ref{increment} to 
$$ R=R_i, \quad \vartheta=\vartheta_i, \quad q=q_i,$$
which in almost exactly the same way as in Case 1 (cf. (\ref{univbound5})) leads to the estimate
\begin{equation} \label{univbound52}
\sup_{0<|h|<h_0} \left | \left | \frac{\delta_h^2 u_1}{ |h|^\gamma} \right | \right |_{L^{q_{i_\infty}}(B_{R_{i_\infty}+4h_0})} \leq \sup_{0<|h|<h_0} \left | \left | \frac{\delta_h^2 u_1}{ |h|^{\beta}} \right | \right |_{L^{q_{i_\infty}}(B_{R_{i_\infty}+4h_0})} \leq C_{13}(n,s,\lambda,\alpha),
\end{equation}
where we used that by (\ref{choice4}) we have $\gamma < 1 \leq \beta = \frac{1}{q_{i_\infty}}+ \vartheta_{i_\infty}$. 
Next, we set $\widetilde \vartheta_i:=\gamma-\frac{1}{q_i}$ and observe that $$\frac{1+\widetilde \vartheta_{i} q_{i}}{q_{i}}=\gamma \in (0,1).$$
Therefore, for $i=i_\infty,...,i_\infty + j_\infty -1$ we can apply Proposition \ref{increment} to 
$$ R=R_i, \quad \vartheta= \widetilde \vartheta_i, \quad q=q_i,$$
so that along with (\ref{relations52}) and the observation that $s > 1/2$ implies
$$ \frac{1+2s+ \widetilde \vartheta_i q_i}{q_i+1} > \frac{2+\widetilde \vartheta_{i} q_{i}}{q_{i}+1} = 1+ \frac{q_i(\gamma-1)}{q_i+1}>\gamma, $$
we obtain the estimates 
$$ \sup_{0<|h|<h_0} \left | \left | \frac{\delta_h^2 u_1}{ |h|^\gamma} \right | \right |_{L^{q_{i+1}}(B_{R_{i+1}+4h_0})} \leq C_{14} \left ( \sup_{0<|h|<h_0} \left | \left | \frac{\delta_h^2 u_1}{ |h|^\gamma} \right | \right |_{L^{q_i}(B_{R_i+4h_0})} + 1 \right ), \text{ } i=i_\infty,...,i_\infty+j_\infty-1,$$
and
$$ \sup_{0<|h|<h_0} \left | \left | \frac{\delta_h^2 u_1}{ |h|^\gamma} \right | \right |_{L^{q_{i_\infty+j_\infty}}(B_{3/4})} \leq C_{14} \left ( \sup_{0<|h|<h_0} \left | \left | \frac{\delta_h^2 u_1}{ |h|^\gamma} \right | \right |_{L^{q_{i_\infty+j_\infty-1}}(B_{R_{i_\infty+j_\infty-1}+4h_0})} + 1 \right ),$$
where $C_{14}=C_{14}(n,s,\lambda,\alpha)$. Combining these estimates with (\ref{univbound52}) and recalling that $\gamma=1-\varepsilon$,
we arrive at
$$ \sup_{0<|h|<h_0} \left | \left | \frac{\delta_h^2 u_1}{ |h|^{1-\varepsilon}} \right | \right |_{L^{q_{i_\infty+j_\infty}}(B_{3/4})} \leq C_{15}(n,s,\lambda,\alpha).$$
By imitating the arguments used to conclude in case 1 (cf. (\ref{Jl}) and (\ref{J1})), which in particular involves applying Lemma \ref{Holderemb} with $\beta=1-\varepsilon$ and $q=q_{i_\infty+j_\infty}$ (which is applicable in view of (\ref{choice7})), we conclude that
$$ [u_1]_{C^\alpha(B_{1/2})} \leq C=C(n,s,\lambda,\alpha)$$
for a different constant $C$ as the one in (\ref{J1}). The desired estimate (\ref{C2sreg5}) now once again simply follows by rescaling, which finishes the proof.
\end{proof}

\section{Higher H\"older regularity by approximation}
We now use an approximation argument inspired by \cite[section 6]{BLS} and \cite{CSa} in order to prove Theorem \ref{C2sreg} and Theorem \ref{C2srega} under full generality. In order to do so, we need the following definition.
\begin{defin}
	Let $0<r<R$ and let $u \in W^{s,2}(B_R) \cap L^1_{2s}(\mathbb{R}^n)$. We say that $v \in W^{s,2}(B_R) \cap L^1_{2s}(\mathbb{R}^n)$ is a weak solution of the problem $$ \begin{cases} \normalfont
		L_{A}^\Phi v = 0 & \text{ in } B_r \\
		v = u & \text{ a.e. in } \mathbb{R}^n \setminus B_r,
	\end{cases} $$ if we have
	$\mathcal{E}_A^\Phi(u,\varphi) = 0$ for any $\varphi \in W_0^{s,2}(B_r)$
	and $v = u \text{ a.e. in } \mathbb{R}^n \setminus B_r$.
\end{defin}

\begin{lem} \label{approxLinf}
	Let $s \in (0,1)$, $\lambda \geq 1$, $q>\frac{n}{2s}$ and $M \geq 1$. Then for any $\tau>0$, there exists some small enough $\delta=\delta(\tau,n,s,\lambda,q,M)>0$ such that the following is true. Assume that $\Phi$ satisfies (\ref{PhiLipschitz}) and (\ref{PhiMonotone}) with respect to $\lambda$, that $A \in \mathcal{L}_0(\lambda)$ and that we have $f \in L^q(B_1)$. Moreover, suppose that $\widetilde A$ is another kernel coefficient of class  $\mathcal{L}_0(\lambda)$ such that
	\begin{equation} \label{boundsxz}
	||A-\widetilde A||_{L^\infty(\mathbb{R}^n \times \mathbb{R}^n)} \leq \delta, \quad ||f||_{L^q(B_1)} \leq \delta,
	\end{equation}
	and let $u \in W^{s,2}(B_1) \cap L^1_{2s}(\mathbb{R}^n)$ be a local weak solution of
	\begin{equation} \label{eq87}
	L_A^\Phi u=f \text{ in } B_1
	\end{equation}
	that satisfies
	\begin{equation} \label{boundsyz}
	\sup_{x \in B_1} |u(x)| + \int_{\mathbb{R}^n \setminus B_1} \frac{|u(y)|}{|y|^{n+2s}}dy \leq M.
	\end{equation}
	Then the unique weak solution $v \in W^{s,2}(B_1) \cap L^1_{2s}(\mathbb{R}^n)$ of the problem
	\begin{equation} \label{constcof3}
	\begin{cases} \normalfont
	L_{\widetilde A}^\Phi v = 0 & \text{ in } B_{7/8} \\
	v = u & \text{ a.e. in } \mathbb{R}^n \setminus B_{7/8}
	\end{cases}
	\end{equation}
	satisfies 
	\begin{equation} \label{appy}
	||u-v||_{L^\infty(B_{3/4})} \leq \tau.
	\end{equation}
\end{lem}
\begin{proof}
First of all, we remark that the existence of a unique weak solution of the problem (\ref{constcof3}) belonging to $W^{s,2}(B_1) \cap L^1_{2s}(\mathbb{R}^n)$ can be shown almost exactly as in \cite[Theroem 1 and Remark 3]{existence} by using the theory of monotone operators and additionally taking into account the bounds (\ref{PhiLipschitz}) and (\ref{PhiMonotone}) imposed on $\Phi$. \newline
We now prove by contradiction. Assume that the conclusion is not true. Then there exist some $\tau>0$, sequences of kernel coefficients $\{A_m\}_{m=1}^\infty$ and $\{\widetilde A_m\}_{m=1}^\infty$ of class $\mathcal{L}_0(\lambda)$, a sequence of functions $\{\Phi_m\}_{m=1}^\infty$ satisfying (\ref{PhiLipschitz}) and (\ref{PhiMonotone}), and sequences $\{u_m\}_{k=1}^\infty \subset W^{s,2}(B_1) \cap L^1_{2s}(\mathbb{R}^n)$, $\{f_m\}_{m=1}^\infty \subset L^q(B_1)$, such that for any $m$ the function $u_m$ is a local weak solution of the problem
\begin{equation} \label{eq87c}
L_{A_m}^{\Phi_m} u_m =f_m \text{ in } B_1,
\end{equation}
\begin{equation} \label{boundsyzc}
\sup_{x \in B_1} |u_m(x)| + \int_{\mathbb{R}^n \setminus B_1} \frac{|u_m(y)|}{|y|^{n+2s}}dy \leq M,
\end{equation}
\begin{equation} \label{boundsxzc}
||A_m-\widetilde A_m||_{L^\infty(\mathbb{R}^n \times \mathbb{R}^n)} \leq \frac{1}{m}, \quad ||f_m||_{L^q(B_1)} \leq \frac{1}{m},
\end{equation}
but for any $m$ the unique weak solution $v_m \in W^{s,2}(B_1) \cap L^1_{2s}(\mathbb{R}^n)$ of
\begin{equation} \label{constcof3c}
\begin{cases} \normalfont
L_{\widetilde A_m}^{\Phi_m} v_m = 0 & \text{ in } B_{7/8} \\
v_m = u_m & \text{ a.e. in } \mathbb{R}^n \setminus B_{7/8}
\end{cases}
\end{equation}
satisfies 
\begin{equation} \label{appyc}
||u_m-v_m||_{L^\infty(B_{3/4})} > \tau.
\end{equation}
In view of (\ref{eq1}), (\ref{PhiMonotone}) and using $w_m:=u_m-v_m \in W^{s,2}_0(B_{7/8})$ as a test function in (\ref{constcof3c}) and also in (\ref{eq87c}), we obtain
\begin{align*}
& \int_{\mathbb{R}^n} \int_{\mathbb{R}^n} \frac{(w_m(x)-w_m(y))^2}{|x-y|^{n+2s}}dydx \\
\leq & \lambda \int_{\mathbb{R}^n} \int_{\mathbb{R}^n} \widetilde A_m(x,y) \frac{((u_m(x)-u_m(y))-(v_m(x)-v_m(y)))^2}{|x-y|^{n+2s}}dydx \\
\leq & \lambda^2 \bigg ( \int_{\mathbb{R}^n} \int_{\mathbb{R}^n} \widetilde A_m (x,y) \frac{\Phi_m(u_m(x)-u_m(y))(w_m(x)-w_m(y))}{|x-y|^{n+2s}}dydx \\
& - \underbrace{\int_{\mathbb{R}^n} \int_{\mathbb{R}^n} \widetilde A_m(x,y) \frac{ \Phi_m(v_m(x)-v_m(y))(w_m(x)-w_m(y))}{|x-y|^{n+2s}}dydx}_{=0} \bigg ) \\
= & \lambda^2 \bigg (\int_{\mathbb{R}^n} \int_{\mathbb{R}^n} (\widetilde A_m(x,y)-A_m(x,y)) \frac{\Phi_m(u_m(x)-u_m(y))(w_m(x)-w_m(y))}{|x-y|^{n+2s}}dydx \\
& + \int_{\mathbb{R}^n} \int_{\mathbb{R}^n} A_m(x,y) \frac{\Phi_m(u_m(x)-u_m(y))(w_m(x)-w_m(y))}{|x-y|^{n+2s}}dydx \bigg )\\
= & \underbrace{\lambda^2 \int_{\mathbb{R}^n} \int_{\mathbb{R}^n} (\widetilde A_m(x,y)-A_m(x,y)) \frac{\Phi_m(u_m(x)-u_m(y))(w_m(x)-w_m(y))}{|x-y|^{n+2s}}dydx}_{=: I_1} \\
& + \underbrace{\lambda^2 \int_{B_1} f_m(x)w_m(x)dx}_{:=I_2}.
\end{align*}
By using (\ref{PhiLipschitz}) and (\ref{boundsxzc}), we further estimate $I_1$ as follows
\begin{align*}
I_1 \leq & \lambda^3 \int_{\mathbb{R}^n} \int_{\mathbb{R}^n} |\widetilde A_m(x,y)-A_m(x,y)| \frac{|u_m(x)-u_m(y)||w_m(x)-w_m(y)|}{|x-y|^{n+2s}}dydx \\
\leq & \lambda^3 ||A_m-\widetilde A_m||_{L^\infty(\mathbb{R}^n \times \mathbb{R}^n)} \int_{\mathbb{R}^n} \int_{\mathbb{R}^n} \frac{|u_m(x)-u_m(y)||w_m(x)-w_m(y)|}{|x-y|^{n+2s}}dydx \\
\leq & \lambda^3 \frac{1}{m} \underbrace{ \int_{B_{15/16}} \int_{B_{15/16}} \frac{|u_m(x)-u_m(y)||w_m(x)-w_m(y)|}{|x-y|^{n+2s}}dydx}_{=: I_{1,1}} \\
& + 2 \lambda^3 \frac{1}{m} \underbrace{ \int_{B_{7/8}} \int_{\mathbb{R}^n \setminus B_{15/16}} \frac{|u_m(x)||w_m(x)|}{|x-y|^{n+2s}}dydx}_{=: I_{1,2}} \\
& + 2 \lambda^3 \frac{1}{m} \underbrace{ \int_{B_{7/8}} \int_{\mathbb{R}^n \setminus B_{15/16}} \frac{|u_m(y)||w_m(x)|}{|x-y|^{n+2s}}dydx}_{=: I_{1,3}}.
\end{align*}
In order to proceed, we observe that since $n>2s$, we have $q > \frac{n}{2s} > \frac{2n}{n+2s}$, so that H\"older's inequality and (\ref{boundsxzc}) yield
\begin{equation} \label{fest}
\left ( \int_{B_1} |f_m(x)|^\frac{2n}{n+2s}dx \right )^\frac{n+2s}{2n} \leq C_1||f_m||_{L^q(B_1)} \leq \frac{C_1}{m},
\end{equation}
where $C_1=C_1(n,s,q)>0$.
By using the Cauchy-Schwarz inequality, Theorem \ref{Cacc}, (\ref{boundsyzc}) and (\ref{fest}), for $I_{1,1}$ we obtain
\begin{align*}
I_{1,1} \leq & \left (\int_{B_{15/16}} \int_{B_{15/16}} \frac{(u_m(x)-u_m(y))^2}{|x-y|^{n+2s}}dydx \right )^\frac{1}{2} \left (\int_{\mathbb{R}^n} \int_{\mathbb{R}^n} \frac{(w_m(x)-w_m(y))^2}{|x-y|^{n+2s}}dydx \right )^\frac{1}{2} \\
\leq & C_2 \left (||u_m||_{L^2(B_1)}^2 + ||u_m||_{L^1(B_1)}\int_{\mathbb{R}^n \setminus B_1} \frac{|u_m(y)|}{|y|^{n+2s}}dy + \left (\int_{B_1} |f_m(x)|^\frac{2n}{n+2s}dx \right )^\frac{n+2s}{n} \right )^\frac{1}{2} \\
& \times \left (\int_{\mathbb{R}^n} \int_{\mathbb{R}^n} \frac{(w_m(x)-w_m(y))^2}{|x-y|^{n+2s}}dydx \right )^\frac{1}{2} \\
\leq & C_3 \left (\int_{\mathbb{R}^n} \int_{\mathbb{R}^n} \frac{(w_m(x)-w_m(y))^2}{|x-y|^{n+2s}}dydx \right )^\frac{1}{2},
\end{align*}
where $C_2$ and $C_3$ depend only on $n,s,\lambda,q$ and $M$. For $I_{1,2}$, by using Lemma \ref{tailestz}, the Cauchy-Schwarz-inequality, the fractional Friedrichs-Poincar\'e inequality (Lemma \ref{Friedrichs}) and (\ref{boundsyzc}), we have
\begin{align*}
I_{1,2} & \leq C_4 \int_{B_{7/8}} \int_{\mathbb{R}^n \setminus B_{15/16}} \frac{|u_m(x)||w_m(x)|}{|y|^{n+2s}}dydx \\
& = C_5 \int_{B_{7/8}} |u_m(x)||w_m(x)|dx \\
& \leq C_5 ||w_m||_{L^2(B_{7/8})} ||u_m||_{L^2(B_{7/8})}
\leq C_6 \left (\int_{\mathbb{R}^n} \int_{\mathbb{R}^n} \frac{(w_m(x)-w_m(y))^2}{|x-y|^{n+2s}}dydx \right )^\frac{1}{2},
\end{align*}
where $C_4=15^{n+2s}$, $C_5=C_5(n,s)>0$ and $C_6=C_6(n,s,M)>0$. Similarly, by using Lemma \ref{tailestz}, the Cauchy-Schwarz-inequality, Lemma \ref{tail}, Lemma \ref{Friedrichs} and (\ref{boundsyzc}), 
for $I_{1,3}$ we obtain
\begin{align*}
I_{1,3} & \leq C_4 \int_{B_{7/8}} \int_{\mathbb{R}^n \setminus B_{15/16}} \frac{|u_m(y)||w_m(x)|}{|y|^{n+2s}}dydx \\
& \leq C_7 ||w_m||_{L^2(B_{7/8})} \int_{\mathbb{R}^n \setminus B_{15/16}} \frac{|u_m(y)|}{|y|^{n+2s}}dy \\
& \leq C_8 \left (\int_{\mathbb{R}^n} \int_{\mathbb{R}^n} \frac{(w_m(x)-w_m(y))^2}{|x-y|^{n+2s}}dydx \right )^\frac{1}{2} \left ( ||u_m||_{L^1(B_{1})}+\int_{\mathbb{R}^n \setminus B_{1}} \frac{|u_m(y)|}{|y|^{n+2s}}dy \right ) \\
& \leq C_9 \left (\int_{\mathbb{R}^n} \int_{\mathbb{R}^n} \frac{(w_m(x)-w_m(y))^2}{|x-y|^{n+2s}}dydx \right )^\frac{1}{2},
\end{align*}
where again all the constants depend only on $n,s$ and $M$.
Next, by using H\"older's inequality, the fractional Sobolev inequality (cf. \cite[Theorem 6.5]{Hitch}) and (\ref{fest}), we estimate $I_2$ in the following way
\begin{align*}
I_2 \leq & \left ( \int_{B_1} |f_m(x)|^\frac{2n}{n+2s}dx \right )^\frac{n+2s}{2n} \left (\int_{B_{1}} |w_m(x)|^\frac{2n}{n-2s} dx \right )^\frac{n-2s}{2n} \\
\leq & C_{10} \frac{1}{m} \left (\int_{\mathbb{R}^n} \int_{\mathbb{R}^n} \frac{(w_m(x)-w_m(y))^2}{|x-y|^{n+2s}}dydx \right )^\frac{1}{2},
\end{align*}
where $C_{10}=C_{10}(n,s,q)>0$. Putting the above estimates together, we arrive at
$$ \left (\int_{\mathbb{R}^n} \int_{\mathbb{R}^n} \frac{(w_m(x)-w_m(y))^2}{|x-y|^{n+2s}}dydx \right )^\frac{1}{2} \leq \frac{C_{11}}{m} $$
for some $C_{11}=C_{11}(n,s,\lambda,q,M)>0$.
Combining this estimate with the fractional Friedrichs-Poincar\'e inequality (Lemma \ref{Friedrichs}) leads to
\begin{equation} \label{L2estx}
||w_m||_{L^2(B_{7/8})} \leq C_{12} \left (\int_{\mathbb{R}^n} \int_{\mathbb{R}^n} \frac{(w_m(x)-w_m(y))^2}{|x-y|^{n+2s}}dydx \right )^\frac{1}{2} \leq C_{12}\frac{C_{11}}{m} \xrightarrow{m \to \infty} 0.
\end{equation}
In other words, we have
\begin{equation} \label{L2lim}
\lim_{m \to \infty} ||u_m-v_m||_{L^2(B_{7/8})} =0.
\end{equation}
In view of Theorem \ref{finnish}, Theorem \ref{finnish1}, the fact that $u_m=v_m$ a.e. in $\mathbb{R}^n \setminus B_{7/8}$ and Lemma \ref{tail}, we have
\begin{align*}
\sup_{x \in \overline B_{3/4}}|v_m(x)| + [v_m]_{C^{\beta}(B_{3/4})} \leq & C_{13} \left (||v_m||_{L^2(B_{7/8})}+ \int_{\mathbb{R}^n \setminus B_{7/8}} \frac{|v_m(y)|}{|y|^{n+2s}}dy \right ) \\
\leq & C_{13} \left (||w_m||_{L^2(B_{7/8})}+ ||u_m||_{L^2(B_{7/8})} + \int_{\mathbb{R}^n \setminus B_{7/8}} \frac{|u_m(y)|}{|y|^{n+2s}}dy \right ) \\
\leq & C_{14} \left (||w_m||_{L^2(B_{7/8})}+ ||u_m||_{L^\infty(B_{1})} + \int_{\mathbb{R}^n \setminus B_{1}} \frac{|u_m(y)|}{|y|^{n+2s}}dy \right ),
\end{align*}
so that in view of (\ref{L2estx}) and (\ref{boundsyzc}) the sequence $\{v_m\}_{m=1}^\infty$ is uniformly bounded in $\overline B_{3/4}$ and has uniformly bounded $C^\beta$ seminorms in $B_{3/4}$, where $\beta=\beta(n,s,\lambda,q)>0$.
Moreover, in view of (\ref{boundsyzc}) and Theorem \ref{finnish1}, the sequence $\{u_m\}_{m=1}^\infty$ is also uniformly bounded in $\overline B_{3/4}$ and has uniformly bounded $C^\beta$ seminorms in $B_{3/4}$.
In particular, the same is also true for the sequence $\{u_m-v_m\}_{m=1}^\infty$. Therefore, by the Arzel\`a-Ascoli theorem, by passing to a subsequence if necessary, we obtain that the sequence $\{u_m-v_m\}_{m=1}^\infty$ converges uniformly in $\overline B_{3/4}$ to some function $h$. Since by (\ref{L2lim}) up to passing to another subsequence we have
$$ u_m-v_m \xrightarrow{m \to \infty} 0 \quad \text{a.e. in } B_{7/8},$$
which by uniqueness of the limit implies that $h=0$ a.e. in $B_{3/4}$, we arrive at
\begin{equation} \label{L2lim1}
\lim_{m \to \infty} ||u_m-v_m||_{L^\infty(B_{3/4})} =0.
\end{equation}
In particular, for $m$ large enough we have 
$$ ||u_m-v_m||_{L^\infty(B_{3/4})} \leq \tau,$$
which contradicts (\ref{appyc}). This finishes the proof.
\end{proof}
Next, we use the above Lemma and Theorem \ref{modC2sreg} in order to prove the desired higher H\"older regularity in the case when $A$ is close enough to a locally translation invariant kernel coefficient.
\begin{prop} \label{nonhomreg1}
	Let  $s \in (0,1)$, $\lambda \geq 1$, $q> \frac{n}{2s}$ and let $\Theta=\min \left \{ 2s-\frac{n}{q},1 \right \}$. Then for any $0<\varepsilon < \Theta$, there exists some small enough $\delta=\delta(\varepsilon,n,s,\lambda,q)>0$ such that the following is true. Assume that $\Phi$ satisfies (\ref{PhiLipschitz}) and (\ref{PhiMonotone}) with respect to $\lambda$, that $A \in \mathcal{L}_0(\lambda)$ and that we have $f \in L^q(B_1)$. Moreover, suppose that there exists a kernel coefficient $\widetilde A \in \mathcal{L}_1(B_1,\lambda)$ such that
	\begin{equation} \label{boundsxzs}
	||A-\widetilde A||_{L^\infty(\mathbb{R}^n \times \mathbb{R}^n)} \leq \delta, \quad ||f||_{L^q(B_1)} \leq \delta.
	\end{equation}
	Then for any local weak solution $u \in W^{s,2}(B_1) \cap L^1_{2s}(\mathbb{R}^n)$ of
	\begin{equation} \label{eq87s}
	L_A^\Phi u=f \text{ in } B_1
	\end{equation}
	that satisfies
	\begin{equation} \label{boundsyzs}
	\sup_{x \in B_1} |u(x)| \leq 1, \quad \int_{\mathbb{R}^n \setminus B_1} \frac{|u(y)|}{|y|^{n+2s}}dy \leq 1,
	\end{equation}
we have $u \in C^{\Theta-\varepsilon}(\overline B_{1/2})$ and
$$[u]_{C^{\alpha}(B_{1/2})} \leq C(n,s,\lambda,q).$$
\end{prop}

\begin{proof}
We divide the proof into two parts. \newline 
\textbf{Step 1: Regularity at the origin.}
In this step, our aim is to prove that for any $0<\varepsilon < \Theta$ and any $0<r<1$, there exists some small enough $\delta>0$ such that if $A$, $\widetilde A$, $f$ and $u$ are as above, then 
\begin{equation} \label{originest}
\sup_{x \in B_r} |u(x)-u(0)| \leq C_1 r^{\Theta-\varepsilon}
\end{equation}
for some constant $C_1=C_1(n,s,\lambda,\varepsilon)>0$. 
In order to accomplish this, we fix some $0<\varepsilon < \Theta$ and observe that it suffices to prove that there exist $0<\rho< \frac{1}{3}$ and $\delta>0$ such that if $A$, $\widetilde A$, $f$ and $u$ are as above, then for any $k \in \mathbb{N}_0$ we have
\begin{equation} \label{sufficientest}
\sup_{x \in B_{\rho^k}} |u(x)-u(0)| \leq 2 \rho^{k(\Theta-\varepsilon)}, \quad \int_{\mathbb{R}^n \setminus B_1} \frac{|u(\rho^k y)-u(0)|}{\rho^{k(\Theta-\varepsilon)} |y|^{n+2s}}dy \leq M_0,
\end{equation}
where $M_0:=1+\int_{\mathbb{R}^n \setminus B_1} \frac{dy}{|y|^{n+2s}}<\infty$.
Indeed, assume that (\ref{sufficientest}) were true. Since for any $0<r<1$ there exists some $k \in \mathbb{N}_0$ such that $\rho^{k+1}<r \leq \rho^k$, by the first inequality in (\ref{sufficientest}) we would arrive at
$$ \sup_{x \in B_r} |u(x)-u(0)| \leq \sup_{x \in B_{\rho^k}} |u(x)-u(0)| \leq 2 \rho^{k(\Theta-\varepsilon)} = \frac{2}{\rho^{\Theta-\varepsilon}} \rho^{(k+1)(\Theta-\varepsilon)} \leq \frac{2}{\rho^{\Theta-\varepsilon}}r^{\Theta-\varepsilon} ,$$
which would prove (\ref{originest}) with $C_1=\frac{2}{\rho^{\Theta-\varepsilon}}$. \newline 
In order to prove (\ref{sufficientest}), we proceed by induction. In the case when $k=0$, (\ref{sufficientest}) is true by the assumptions (\ref{boundsyzs}). \newline
Next, suppose that (\ref{sufficientest}) holds up to $k$ and let us prove that it is also true for $k+1$. Let $\tau>0$ to be chosen small enough and consider the corresponding $\delta=\delta(\tau,n,s,\lambda,q,M)>0$ given by Lemma \ref{approxLinf}, where $M:=2+M_0$. Assume that (\ref{boundsxzs}) is satisfied with respect to this $\delta$. Furthermore, define
$$ w_k(x):=\frac{u(\rho^kx)-u(0)}{\rho^{k(\Theta-\varepsilon)}}, \quad f_k(x):=\rho^{k(2s-(\Theta-\varepsilon))} f(\rho^k x)$$
and
$$ A_k(x,y):=A(\rho^kx,\rho^ky), \quad \widetilde A_k(x,y):=\widetilde{A}(\rho^kx,\rho^ky), \quad \Phi_k(t):=\frac{1}{\rho^{k(\Theta-\varepsilon)}} \Phi(\rho^{k(\Theta-\varepsilon)}t) .$$
We note that $A_k \in \mathcal{L}_0(\lambda)$, $\widetilde A_k \in \mathcal{L}_1 \left(\lambda,B_{\frac{1}{\rho^k}} \right ) \subset \mathcal{L}_1(\lambda,B_1)$ and that $\Phi_k$ satisfies (\ref{PhiLipschitz}) and (\ref{PhiMonotone}) with respect to $\lambda$. Moreover, $w_k$ belongs to $W^{s,2}(B_1) \cap L^1_{2s}(\mathbb{R}^n)$ and is a local weak solution of $L_{A_k}^{\Phi_k} w_k = f_k$ in $B_1$, while by (\ref{boundsxzs}) we have
$$||A_k-\widetilde A_k||_{L^\infty(\mathbb{R}^n \times \mathbb{R}^n)} = ||A-\widetilde A||_{L^\infty(\mathbb{R}^n \times \mathbb{R}^n)}\leq \delta $$
and
$$||f_k||_{L^q(B_1)} = \rho^{k(2s-(\Theta-\varepsilon))} \rho^{-k\frac{n}{q}} ||f||_{L^q(B_{\rho^k})} \leq ||f||_{L^q(B_1)} \leq \delta,$$
where we have also used that $\Theta \leq 2s-\frac{n}{q}$ and thus $k \left (2s-(\Theta-\varepsilon)-\frac{n}{q} \right ) \geq k\varepsilon \geq 0$. Moreover, by the induction hypothesis we have
\begin{equation} \label{boundswk}
||w_k||_{L^\infty(B_1)} \leq 2, \quad \int_{\mathbb{R}^n \setminus B_1} \frac{|w_k(y)|}{|y|^{n+2s}}dy \leq M_0.
\end{equation}
Therefore, by Lemma \ref{approxLinf} the unique weak solution $v_k \in W^{s,2}(B_1) \cap L^1_{2s}(\mathbb{R}^n)$ of
$$
\begin{cases} \normalfont
L_{\widetilde A_k}^{\Phi_k} v_k = 0 & \text{ in } B_{7/8} \\
v_k = w_k & \text{ a.e. in } \mathbb{R}^n \setminus B_{7/8}
\end{cases}
$$
satisfies
\begin{equation} \label{approxwk}
||w_k-v_k||_{L^\infty(B_{3/4})} \leq \tau.
\end{equation}
Together with the fact that $w_k(0)=0$, we obtain that for any $x \in B_{1/3}$ we have
\begin{equation} \label{B12}
\begin{aligned}
|w_k(x)| \leq & |w_k(x)-v_k(x)| + |v_k(0)-w_k(0)| + |v_k(x)-v_k(0)| \\
\leq & 2\tau+ [v_k]_{C^{\Theta-\varepsilon/2}(B_{1/3})} |x|^{\Theta-\varepsilon/2}.
\end{aligned}
\end{equation}
Our next goal is to prove that the right-hand side of the previous estimate is uniformly bounded by a constant that does not depend on $k$. In order to do so, we observe that since $\widetilde A_k \in \mathcal{L}_1(\lambda,B_1) \subset \mathcal{L}_1(\lambda,B_{2/3})$, by Theorem \ref{modC2sreg} we have
\begin{equation} \label{B14}
[v_k]_{C^{\Theta-\varepsilon/2}(B_{1/3})} \leq C_2 \left (||v_k||_{L^\infty(B_{2/3})}+ [v_k]_{W^{s,2}(B_{2/3})} + \int_{\mathbb{R}^n \setminus B_{2/3}} \frac{|v_k(y)|}{|y|^{n+2s}}dy \right ), 
\end{equation}
where $C_2=C_2(n,s,\lambda,\Theta,\varepsilon)>0$.
For the first term of the right-hand side, in view of (\ref{approxwk}) and (\ref{boundswk}) we have
$$ ||v_k||_{L^\infty(B_{2/3})} \leq ||v_k||_{L^\infty(B_{3/4})} \leq ||v_k-w_k||_{L^\infty(B_{3/4})} + ||w_k||_{L^\infty(B_{3/4})} \leq \tau + 2.$$
In order to estimate the tail term, we observe that by the same argument used in order to obtain (\ref{L2estx}), we have
$$ ||v_k-w_k||_{L^2(B_{7/8})} \leq C_3 \delta,$$
where $C_3=C_3(n,s,\lambda,q)>0$.
Together with the fact that $v_k=w_k$ in $\mathbb{R}^n \setminus B_{7/8}$, Lemma \ref{tail} and (\ref{boundswk}), we deduce
\begin{align*}
\int_{\mathbb{R}^n \setminus B_{2/3}} \frac{|v_k(y)|}{|y|^{n+2s}}dy \leq & \int_{\mathbb{R}^n \setminus B_{2/3}} \frac{|w_k(y)|}{|y|^{n+2s}}dy+\int_{B_{7/8} \setminus B_{2/3}} \frac{|v_k(y)-w_k(y)|}{|y|^{n+2s}}dy \\
\leq & C_4 \left ( ||w_k||_{L^1(B_{1})} + \int_{\mathbb{R}^n \setminus B_{1}} \frac{|w_k(y)|}{|y|^{n+2s}}dy + \int_{B_{7/8}} |v_k(y)-w_k(y)|dy \right ) \\ 
\leq & C_5 \left ( ||w_k||_{L^\infty(B_{1})}+ \int_{\mathbb{R}^n \setminus B_{1}} \frac{|w_k(y)|}{|y|^{n+2s}}dy+||v_k-w_k||_{L^2(B_{7/8})} \right ) \leq C_6,
\end{align*}
where $C_6=C_6(n,s,\lambda,q,\delta)>0$.
Finally, for the Sobolev seminorm by Theorem \ref{Cacc} and the above estimates we have
\begin{align*}
[v_k]_{W^{s,2}(B_{2/3})} \leq C_7 \left ( ||v_k||_{L^\infty(B_{3/4})}^2 + ||v_k||_{L^\infty(B_{3/4})}\int_{\mathbb{R}^n \setminus B_{3/4}} \frac{|v_k(y)|}{|y|^{n+2s}}dy \right ) \leq C_8, 
\end{align*}
where $C_7$ and $C_8$ do not depend on $k$. By combining the above estimates with (\ref{B12}) and (\ref{B14}), we obtain that for any $x \in B_{1/3}$ we have
\begin{equation} \label{D15}
|w_k(x)| \leq 2 \tau + C_9 |x|^{\Theta-\varepsilon/2}, 
\end{equation}
where again $C_9$ does not depend on $k$.
Next, define 
$$ w_{k+1}(x):=\frac{u(\rho^{k+1}x)-u(0)}{\rho^{{k+1}(\Theta-\varepsilon)}} = \frac{w_k(\rho x)}{\rho^{\Theta-\varepsilon}}. $$
By choosing $\tau$ small enough such that $2\tau<\rho^\Theta$, in view of (\ref{D15}), we obtain 
\begin{equation} \label{wk+1}
|w_{k+1}(x)| \leq 2 \tau \rho^{\varepsilon-\Theta} + C_9 \rho^{\varepsilon-\Theta} |\rho x|^{\Theta-\varepsilon/2} \leq (1+C_9 |x|^{\Theta-\varepsilon/2}) \rho^{\varepsilon /2} \quad \forall x \in B_\frac{1}{3 \rho}.
\end{equation}
In particular, by choosing $\rho$ small enough such that
$\rho \leq (1+C_9)^{-\frac{2}{\varepsilon}}$ and recalling that $\rho <1/3$,
we arrive at $||w_{k+1}||_{L^\infty(B_1)} \leq 1$. By definition of $w_{k+1}$ this is equivalent to 
$$ \sup_{x \in B_{\rho^{k+1}}} |u(x)-u(0)| \leq \rho^{(k+1)(\Theta-\varepsilon)},$$
which proves the first estimate in (\ref{sufficientest}) for $k+1$. \newline
In order to prove the second estimate in (\ref{sufficientest}) for $k+1$, we observe that (\ref{wk+1}) implies
\begin{align*}
\int_{B_\frac{1}{3\rho} \setminus B_1} \frac{|w_{k+1}(y)|}{|y|^{n+2s}}dy \leq & \rho^{\varepsilon/2} \int_{B_\frac{1}{3\rho} \setminus B_1} \frac{1+C_9 |y|^{\Theta-\varepsilon/2}}{|y|^{n+2s}}dy \\
\leq & (1+C_9) \rho^{\varepsilon/2} \int_{B_\frac{1}{3\rho} \setminus B_1} \frac{dy}{|y|^{n+2s+\varepsilon/2-\Theta}} \leq C_{10} \rho^{\varepsilon/2},
\end{align*}
where $C_{10}:=(1+C_9) \int_{\mathbb{R}^n \setminus B_1} \frac{dy}{|y|^{n+2s+\varepsilon/2-\Theta}}<\infty$ does not depend on $k$ and is finite because $2s+\varepsilon/2-\Theta \geq \frac{n}{q}+\varepsilon/2>0$.
Furthermore, by using a change of variables and the first bound in (\ref{boundswk}), we obtain
\begin{align*}
\int_{B_\frac{1}{\rho} \setminus B_\frac{1}{3\rho}} \frac{|w_{k+1}(y)|}{|y|^{n+2s}}dy = & \rho^{\varepsilon-\Theta+2s} \int_{B_1 \setminus B_{1/3}} \frac{|w_{k}(y)|}{|y|^{n+2s}}dy \leq 2 \rho^{\varepsilon/2} \int_{B_1 \setminus B_{1/3}} \frac{dy}{|y|^{n+2s}} \leq C_{11} \rho^{\varepsilon/2},
\end{align*}
where $C_{11}:=3^{n+2s} 2 |B_1|<\infty$.
Moreover, again by a change of variables and the second bound in (\ref{boundswk}), we deduce
$$ \int_{\mathbb{R}^n \setminus B_\frac{1}{\rho}} \frac{|w_{k+1}(y)|}{|y|^{n+2s}}dy = \rho^{\varepsilon-\Theta+2s} \int_{\mathbb{R}^n \setminus B_1} \frac{|w_{k}(y)|}{|y|^{n+2s}}dy \leq M_0 \rho^{\varepsilon/2}. $$
Note that in the last two estimates we also used that $\rho<1$ and that  $\varepsilon-\Theta+2s \geq \varepsilon/2 $. By combining the last three displays and choosing $\rho$ small enough such that 
$$ (C_{10}+C_{11}+M_0)\rho^{\varepsilon/2} \leq M_0,$$
we arrive at
$$ \int_{\mathbb{R}^n \setminus B_1} \frac{|w_{k+1}(y)|}{|y|^{n+2s}}dy \leq (C_{10}+C_{11}+M_0)\rho^{\varepsilon/2} \leq M_0, $$
which proves the second estimate in (\ref{sufficientest}) for $k+1$.
Therefore, for
$$ \rho < \min \left \{\frac{1}{3},(1+C_9)^{-\frac{2}{\varepsilon}},M_0^\frac{2}{\varepsilon} (C_{10}+C_{11}+M_0)^{-\frac{2}{\varepsilon}} \right \}, \quad \tau < \frac{\rho^\Theta}{2} $$
(\ref{sufficientest}) is true for any $k \in \mathbb{N}_0$, which in particular also proves (\ref{originest}) under the assumptions (\ref{boundsxzs}) and (\ref{boundsyzs}), where $\delta$ is chosen as above. \newline
\textbf{Step 2: Regularity in a ball.}
Next, we show the desired higher H\"older regularity in the whole ball $B_{1/2}$. We fix some $0<\varepsilon<\Theta$ and take the corresponding small enough $\delta$ from step 1. Fix $z \in B_{1/2}$, set $L:=2^{n+1}(1+|B_1|)$ and define 
$$ u_{z}(x):=u \left (\frac{x}{2}+z \right )/L, \quad f_{z}(x):= \frac{2^{-2s}}{L} f \left (\frac{x}{2} +z \right )$$
and
$$ A_{z}(x,y):=A \left (\frac{x}{2} +z,\frac{y}{2} +z \right ), \quad \widetilde A_{z}(x,y):=\widetilde{A}\left (\frac{x}{2} +z,\frac{y}{2} +z \right ), \quad \Phi_L(t):=\frac{1}{L} \Phi(L t) .$$
We note that $A_{z} \in \mathcal{L}_0(\lambda)$, $\widetilde A_{z} \in \mathcal{L}_1 \left(\lambda,B_1 \right )$ and that $\Phi_L$ satisfies (\ref{PhiLipschitz}) and (\ref{PhiMonotone}) with respect to $\lambda$. Moreover, $u_{z}$ is a local weak solution of $L_{A_{z}}^{\Phi_L} u_{z} = f_{z}$ in $B_1$ and by (\ref{boundsxzs}) we have
$$||A_z-\widetilde A_z||_{L^\infty(\mathbb{R}^n \times \mathbb{R}^n)} = ||A-\widetilde A||_{L^\infty(\mathbb{R}^n \times \mathbb{R}^n)} \leq \delta $$
and
$$||f_z||_{L^q(B_1)} = \frac{2^{n/q-2s}}{L} ||f||_{L^q(B_{1/2}(z))} \leq ||f||_{L^q(B_1)} \leq \delta.$$
Additionally, by (\ref{boundsyzs}) we have
$$\sup_{x \in B_1} |u_z(x)| \leq \sup_{x \in B_{1/2}(z)} |u(x)| \leq \sup_{x \in B_1} |u(x)| \leq 1 $$ and together with Lemma \ref{tail}
\begin{align*}
\int_{\mathbb{R}^n \setminus B_1} \frac{|u_z(y)|}{|y|^{n+2s}}dy = & \frac{2^{-2s}}{L} \int_{\mathbb{R}^n \setminus B_{1/2}(z)} \frac{|u(y)|}{|y-z|^{n+2s}}dy \\
\leq & \frac{2^n}{L} \int_{\mathbb{R}^n \setminus B_1} \frac{|u(y)|}{|y|^{n+2s}}dy + \frac{2^n}{L} ||u||_{L^1(B_1)} \\
\leq & \frac{2^n}{L} \int_{\mathbb{R}^n \setminus B_1} \frac{|u(y)|}{|y|^{n+2s}}dy + \frac{2^n |B_1|}{L} ||u||_{L^\infty(B_1)} \leq 1.
\end{align*}
Therefore, we are in the position to apply step 1 to $u_z$, which yields
$$ \sup_{x \in B_r} |u_z(x)-u_z(0)| \leq C_1 r^{\Theta-\varepsilon}, \quad 0<r<1.$$
By rewriting this estimate in terms of $u$, for any $z \in B_{1/2}$ we obtain
\begin{equation} \label{est14}
\sup_{x \in B_r(z)} |u(x)-u(z)| \leq C_1 L r^{\Theta-\varepsilon}, \quad 0<r<\frac{1}{2}.
\end{equation}
Now fix two points $x,y \in B_{1/2}$. Then applying (\ref{est14}) with $r=\frac{|x-y|}{2}<1/2$ and $z=(x+y)/2$ yields
\begin{align*}
|u(x)-u(y)| \leq |u(x)-u(z)|+|u(y)-u(z)| \leq & 2 \sup_{\omega \in B_r(z)} |u(w)-u(z)| \\
\leq & 2C_1 L r^{\Theta-\varepsilon} \leq 2C_1 L |x-y|^{\Theta-\varepsilon},
\end{align*}
which proves the desired H\"older regularity of $u$.
\end{proof}
In order to obtain the estimate (\ref{Hoeldest}) in our main results with its precise scaling, we now first prove Theorem \ref{C2srega} at scale 1 by using scaling and covering arguments. The general case will then follow by another scaling argument.
\begin{thm} \label{C2sregy}
	Let $\lambda \geq 1$ and $f \in L^q(B_1)$ for some $q>\frac{n}{2s}$.
	Consider a kernel coefficient $A \in \mathcal{L}_0(\lambda)$ and suppose that $\Phi$ satisfies (\ref{PhiLipschitz}) and (\ref{PhiMonotone}) with respect to $\lambda$. Fix some $0<\alpha<\min \big \{2s-\frac{n}{q},1 \big\}$. Then there exists some small enough $\delta=\delta(\alpha,n,s,\lambda,q)>0$, such that if for any $z \in B_1$, there exists some small enough radius $r_{z}>0$ and some $A_{z} \in \mathcal{L}^1(\lambda,B_{r_{z}}(z))$ such that $$||A-A_{z}||_{L^\infty(B_{r_z}(z) \times B_{r_z}(z))} \leq \delta,$$
	then for any local weak solution $u \in W^{s,2}(B_1) \cap L^1_{2s}(\mathbb{R}^n)$ of the equation $L_{A}^\Phi u = f$ in $B_1$, we have $u \in C^\alpha(\overline B_\sigma)$ and
	\begin{equation} \label{Hoeldest4}
	[u]_{C^\alpha(B_{\sigma})} \leq C \bigg ( ||u||_{L^2(B_1)} + \int_{\mathbb{R}^n \setminus B_{1}} \frac{|u(y)|}{|y|^{n+2s}}dy +  ||f||_{L^q(B_1)} \bigg ),
	\end{equation}
	where $C=C(n,s,\lambda,\alpha,\sigma,q,\{r_z\}_{z \in B_1})>0$.
\end{thm}
\begin{proof}
Fix $\alpha \in (0,\Theta)$, where as before $\Theta=\min \left \{ 2s-\frac{n}{q},1 \right \}$, set $\varepsilon = \Theta-\alpha$ and let $\delta=\delta(\varepsilon,n,s,\lambda,q)>0$ be given by Proposition \ref{nonhomreg1}. We need to prove that $u \in C^{\Theta-\varepsilon}_{loc}(\overline B_\sigma)$. Let $\delta=\delta(\varepsilon,n,s,\lambda,q)>0$ be the corresponding $\delta$ given by Proposition \ref{nonhomreg1} and fix some $\sigma \in (0,1)$.
Fix some $z \in \overline B_{\sigma}$. Then by assumption, there exists some small enough radius $r_z \in (0,1)$ with $B_{2r_z}(z) \subset B_1$ and some kernel coefficient $A_z \in \mathcal{L}_1(\lambda,B_{r_z}(z))$ such that
$$
	||A- A_z||_{L^\infty(B_{r_z}(z) \times B_{r_z}(z))} \leq \delta.
$$
Then the kernel coefficient
$$ \widetilde A (x,y) := \begin{cases} \normalfont
	A_z(x,y) & \text{if } (x,y) \in B_{r_z}(z) \times B_{r_z}(z) \\
	A(x,y) & \text{if } (x,y) \notin B_{r_z}(z) \times B_{r_z}(z)
\end{cases} $$
also belongs to $\mathcal{L}_1(\lambda,B_{r_z}(z))$ and satisfies
\begin{equation} \label{Asmall8}
	||A- \widetilde A||_{L^\infty(\mathbb{R}^n \times \mathbb{R}^n)} \leq \delta.
\end{equation}
In the case when $u \equiv 0$, the desired H\"older regularity trivially holds. Otherwise, set
$$ M_{z}:= \sup_{x \in B_{r_z}(z)} |u(x)| + r_z^{2s} \int_{\mathbb{R}^n \setminus B_{r_z}(z)} \frac{|u(y)|}{|z-y|^{n+2s}}dy + \frac{r_z^{2s-n/q}}{\delta} ||f||_{L^q(B_{r_z}(z))}>0. $$ Consider the scaled functions $u_1 \in W^{s,2}(B_{1}) \cap L^1_{2s}(\mathbb{R}^n)$ and $f_1 \in L^q(B_1)$ given by
$$ u_1(x):=\frac{1}{M_{z}}u(r_z x+z), \quad f_1(x):=\frac{r^{2s}}{M_z} f(r_z x+z)$$
and also 
$$A_1(x,y):= A(r_z x+z,r_z y+z), \quad \widetilde A_1(x,y):= \widetilde A(r_z x+z,r_z y+z), \quad \Phi_1(t):= \frac{1}{M_{z}} \Phi(M_{z} t).$$
We note that $u_1$ is a local weak solution of $L_{A_1}^{\Phi_1} u_1 = f_1$ in $B_1$. Moreover, observe that $A_1 \in \mathcal{L}_0(\lambda)$ and $\widetilde A_1 \in \mathcal{L}_1(\lambda,B_1)$, while $\Phi_1$ satisfies (\ref{PhiLipschitz}) and (\ref{PhiMonotone}) with respect to $\lambda$. Furthermore, by using changes of variables it is easy to verify that $u_1$ and $f_1$ satisfy
\begin{equation} \label{smallu2}
\sup_{x \in B_1} |u_1(x)| \leq 1, \quad \int_{\mathbb{R}^n \setminus B_1} \frac{|u_1(y)|}{|y|^{n+2s}}dy \leq 1, \quad ||f_1||_{L^q(B_1)} \leq \delta,
\end{equation}
while (\ref{Asmall8}) implies that
\begin{equation} \label{Asmall9}
||A_1-\widetilde A_1||_{L^\infty(\mathbb{R}^n \times \mathbb{R}^n)} \leq \delta.
\end{equation}
Therefore, in view of (\ref{smallu2}) and (\ref{Asmall9}) the assumptions (\ref{boundsxzs}) and (\ref{boundsyzs}) from Proposition \ref{nonhomreg1} are verified with respect to $u_1,f_1,A_1$ and $\widetilde A_1$, so that by Proposition \ref{nonhomreg1} we obtain 
$$[u_1]_{C^{\Theta-\varepsilon}(B_{1/2})} \leq C_1(n,s,\lambda,q).$$
By rescaling and then using Theorem \ref{finnish}, we arrive at the estimate
\begin{equation} \label{preestx}
\begin{aligned}
[u]_{C^{\Theta-\varepsilon}(B_{r_z/2}(z))} \leq & \frac{C_1}{r_z^{\Theta-\varepsilon}} \bigg ( \sup_{x \in B_{r_z}(z)} |u(x)| + r_z^{2s} \int_{\mathbb{R}^n \setminus B_{r_z}(z)} \frac{|u(y)|}{|z-y|^{n+2s}}dy \\ & + \frac{r_z^{2s-n/q}}{\delta} ||f||_{L^q(B_{r_z}(z))} \bigg )\\
\leq & \frac{C_2}{r_z^{\Theta-\varepsilon}} \bigg ( r_z^{-n/2}||u||_{L^2(B_{2r_z}(z))} + r_z^{2s} \int_{\mathbb{R}^n \setminus B_{r_z}(z)} \frac{|u(y)|}{|z-y|^{n+2s}}dy \\ & + r_z^{2s-n/q} ||f||_{L^q(B_{2r_z}(z))} \bigg ),
\end{aligned}
\end{equation}
where $C_2=C_2(n,s,\lambda,q,\Theta,\varepsilon)>0$. Since $\left \{B_{r_z/4}(z) \right \}_{z \in \overline B_\sigma}$ is an open covering of $\overline B_{\sigma}$ and $\overline B_{\sigma}$ is compact, there exists a finite subcover $\left \{B_{r_{z_i}/4}(z_i) \right \}_{i=1}^N$ of $\overline B_{\sigma}$ and hence of $B_{\sigma}$. Set $$r_{\textnormal{min}}:= \min_{i=1,...,N}r_{z_i}>0.$$ 
Fix $x,y \in B_\sigma$ with $x \neq y$. Then $x \in B_{r_{z_i}/4}(z_i)$ for some $i=1,...,N$.
If $|x-y| < r_{\textnormal{min}}/4$, then in particular $y \in B_{r_{z_i}/2}(z_i)$, so that by (\ref{preestx}) and Lemma \ref{tail} we have
\begin{align*}
\frac{|u(x)-u(y)|}{|x-y|^{\Theta-\varepsilon}} \leq &  [u]_{C^{\Theta-\varepsilon}(B_{r_{z_i}/2}(z))} \\ \leq & C_3 \bigg ( ||u||_{L^2(B_{2r_{z_i}}(z_i))} + \int_{\mathbb{R}^n \setminus B_{r_{z_i}}(z_i)} \frac{|u(y)|}{|z_i-y|^{n+2s}}dy + ||f||_{L^q(B_{2r_{z_i}}(z_i))} \bigg ) \\
\leq & C_4 \bigg ( ||u||_{L^2(B_1)} + \int_{\mathbb{R}^n \setminus B_{1}} \frac{|u(y)|}{|y|^{n+2s}}dy + ||f||_{L^q(B_1)} \bigg ),
\end{align*}
where $C_3$ and $C_4$ depend only on $n,s,\lambda,q,\Theta,\varepsilon$ and $r_{\textnormal{min}}$.
If $|x-y| \geq r_{\textnormal{min}}/4$, then in view of Theorem \ref{finnish} and Lemma \ref{tail}, we have
\begin{align*}
\frac{|u(x)-u(y)|}{|x-y|^{\Theta-\varepsilon}} \leq & 2 \left (\frac{4}{r_{\textnormal{min}}} \right )^{\Theta-\varepsilon} \sup_{x \in B_\sigma} |u(x)| \\
\leq & C_5 \left (||u||_{L^2(B_1)} + \int_{\mathbb{R}^n \setminus B_{\sigma}} \frac{|u(y)|}{|y|^{n+2s}}dy + ||f||_{L^q(B_1)} \right ) \\
\leq & C_6 \left (||u||_{L^2(B_1)} + \int_{\mathbb{R}^n \setminus B_{1}} \frac{|u(y)|}{|y|^{n+2s}}dy + ||f||_{L^q(B_1)} \right ),
\end{align*}
where $C_5$ and $C_6$ depend only on $n,s,\lambda,q,\Theta,\varepsilon,\sigma$ and $r_{\textnormal{min}}$.
Recalling that $\alpha=\Theta-\varepsilon$, combining the above estimates now proves the estimate (\ref{Hoeldest4}) and in particular $u \in C^\alpha(\overline B_\sigma)$.
\end{proof}

\begin{proof}[Proof of Theorem \ref{C2srega}]
Fix some $0<\alpha<\min \big \{2s-\frac{n}{q},1 \big\}$ and let $\delta=\delta(\alpha,n,s,\lambda,q)>0$ be given by Theorem \ref{C2sregy}. Fix $x_0 \in \Omega$ and $R>0$ such that $B_R(x_0) \Subset \Omega$, so that by assumption for any $z^\prime \in B_R(x_0)$, there is some small enough radius $r_{z^\prime}>0$ and some $A_{z^\prime} \in \mathcal{L}^1(\lambda,B_{r_{z^\prime}}(z^\prime))$ such that
$$
	||A-A_{z^\prime}||_{L^\infty(B_{r_{z^\prime}}(z^\prime) \times B_{r_{z^\prime}}(z^\prime))} \leq \delta.
$$
Consider the functions $u_1 \in W^{s,2}(B_{1}) \cap L^1_{2s}(\mathbb{R}^n)$ and $f_1 \in L^q(B_1)$ given by
$$ u_1(x):=u(Rx+x_0), \quad f_1(x):=R^{2s} f(Rx+x_0)$$
and also 
$$A_1(x,y):= A(Rx+x_0,Ry+x_0), \quad (A_1)_z(x,y):= A_{Rz+x_0}(Rx+x_0,Ry+x_0), \text{ } z \in B_1,$$ 
where $A_{Rz+x_0}$ exists for any $z \in B_1$ since in this case we have $Rz+x_0 \in B_R(x_0)$.
We note that for any $z \in B_1$ and $r_z:=r_{Rz+x_0}/R>0$, we have $(A_1)_{z} \in \mathcal{L}^1(\lambda,B_{r_z}(z))$ and $$||A_1-(A_1)_{z}||_{L^\infty(B_{r_z}(z) \times B_{r_z}(z))} \leq \delta.$$ In addition, $u_1$ is a local weak solution of $L_{A_1}^{\Phi} u_1 = f_1$ in $B_1$. Therefore, by Theorem \ref{C2sregy} along with some changes of variables, for any $\sigma \in (0,1)$ we obtain the estimate
\begin{align*}
	R^{\alpha}[u]_{C^\alpha(B_{\sigma R}(x_0))} = & [u_1]_{C^\alpha(B_{\sigma})} \\
	\leq & C \bigg ( ||u_1||_{L^2(B_1)} + \int_{\mathbb{R}^n \setminus B_{1}} \frac{|u_1(y)|}{|y|^{n+2s}}dy +  ||f_1||_{L^q(B_1)} \bigg ) \\ = & C \bigg ( R^{-\frac{n}{2}} ||u||_{L^2(B_R(x_0))} + R^{2s} \int_{\mathbb{R}^n \setminus B_{R}(x_0)} \frac{|u(y)|}{|x_0-y|^{n+2s}}dy \\ & + R^{2s-\frac{n}{q}} ||f||_{L^q(B_R(x_0))} \bigg ),
\end{align*}
which proves the estimate (\ref{Hoeldest}). Furthermore, since $x_0 \in \Omega$ is arbitrary, we in particular obtain that $u \in C^\alpha_{loc}(\Omega)$.
\end{proof}	

\begin{proof}[Proof of Theorem \ref{C2sreg}]
Fix some $0<\alpha<\min \big \{2s-\frac{n}{q},1 \big\}$ and let $\delta=\delta(\alpha,n,s,\lambda,q)>0$ be given by Theorem \ref{C2srega}.
Fix some $R>0$ and some $x_0 \in \Omega$ with $B_R(x_0) \Subset \Omega$. Since $A$ satisfies (\ref{contkernel}) in $\Omega$ with respect to some $\varepsilon>0$, there exists some small enough $r_\delta>0$ such that
\begin{equation} \label{Asmall5}
	\sup_{\substack{_{x,y \in B_R(x_0)}\\{|x-y| \leq \varepsilon}}} |A(x+h,y+h)-A(x,y)| \leq \delta \quad \forall h \in B_{r_\delta}. 
\end{equation}
Now fix some $z \in B_R(x_0)$ and some small enough radius $r_z \in (0,1)$ such that $r_z \leq \min\{\varepsilon/2,r_\delta\}$ and $B_{r_z}(z) \subset B_R(x_0)$. Then for all $x,y \in B_{r_z}(z)$ we have $z-y \in B_{r_{\delta}}$ and $z-x \in B_{r_{\delta}}$, so that (\ref{Asmall5}) implies
$$ \sup_{x,y \in B_{r_z}(z)} |A(x-y+z,z)-A(x,y)| \leq \delta, \quad \sup_{x,y \in B_{r_z}(z)} |A(z,y-x+z)-A(x,y)| \leq \delta. $$
Therefore, by additionally taking into account the symmetry of $A$, we see that the kernel coefficient defined by
$$ A_z(x,y) :=
	\frac{1}{2} \left (A(x-y+z,z)+A(y-x+z,z) \right ) $$
satisfies 
$$ ||A-A_z||_{L^\infty(B_{r_z}(z) \times B_{r_z}(z))} \leq \delta
$$
and clearly belongs to the class $\mathcal{L}_1(\lambda,B_{r_z}(z))$. Since $z \in B_R(x_0)$ is arbitrary, all assumptions from Theorem \ref{C2srega} are satisfied with $\Omega$ replaced by $B_R(x_0)$. Therefore, by Theorem \ref{C2srega} we see that the estimate (\ref{Hoeldest}) holds in any ball $B_R(x_0) \Subset \Omega$. In addition, since $x_0 \in \Omega$ is arbitrary, we obtain that $u \in C^\alpha_{loc}(\Omega)$.
\end{proof}

\bibliographystyle{amsplain}

\end{document}